\documentclass[11pt,final]{siamltex}
\usepackage{amsmath}
\usepackage{bm}
\usepackage{amssymb,version}
\usepackage{cases}
\usepackage{color}
\usepackage{verbatim}
\usepackage{graphicx}
\newtheorem{rem}{Remark}[section]
\usepackage{subfigure}
\allowdisplaybreaks
\begin{document}
\graphicspath{{figures/},}
    \title{On a class of higher-order length preserving and energy decreasing IMEX schemes for the Landau-Lifshitz equation
\thanks{This work is supported by the National Natural Science Foundation of China  grants  12271302, 11971407 and 12131014, and by the Hong Kong Polytechnic University Postdoctoral Research Fund 1-W22P.}}
 \author{ Xiaoli Li
        \thanks{School of Mathematics, Shandong University, Jinan, Shandong, 250100, P.R. China. Email: xiaolimath@sdu.edu.cn}.
        \and Nan Zheng
         \thanks{Department of Applied Mathematics, The Hong Kong Polytechnic University, Hung Hom, Kowloon, Hong Kong. Email: znan2017@163.com}.
         \and Jie Shen
        \thanks{Corresponding Author. School of Mathematical Science, Eastern Institute of Technology, Ningbo, China. Email: jshen@eitech.edu.cn}.
}

  \maketitle

\begin{abstract}
 We construct new higher-order implicit-explicit (IMEX) schemes using the generalized scalar auxiliary variable (GSAV) approach for the Landau-Lifshitz equation. These schemes are linear, length preserving and only require solving one elliptic equation with constant coefficients at each time step. We show that numerical solutions of these schemes are uniformly bounded without any restriction on the time step size, and establish rigorous error estimates in $l^{\infty}(0,T;H^1(\Omega)) \bigcap l^{2}(0,T;H^2(\Omega))$ of orders 1 to 5 in a  unified framework. 
 \end{abstract}

  
 \section{Introduction}
 Magnetization dynamics in a ferromagnetic material occupying a region $\Omega$ is governed by the Landau-Lifshitz equation \cite{guo1993landau,landau1992theory}: 
   \begin{equation}\label{e_original model}
 \frac{\partial \textbf{m}}{\partial t} = -\beta \textbf{m} \times \Delta \textbf{m}
- \gamma \textbf{m} \times ( \textbf{m} \times \Delta \textbf{m} ), \quad  \ {\rm in} \ \Omega \times J, 
\end{equation}
with 
 \begin{equation}\label{e_initial condition}
  \textbf{m}(\textbf{x},0 ) = \textbf{m}_0(\textbf{x}) , \text{ with } |\textbf{m}_0(\textbf{x}) |=1, \ {\rm in}\ \Omega,
\end{equation} 
subject to either homogeneous Neumann or periodic boundary conditions.
In the above, $\textbf{m}=(m_1,m_2,m_3)^t$ describes the magnetization in continuum ferromagnets, $\Omega$ is an open bounded domain in $\mathbb{R}^d$ with $d \in \{ 1,2,3\}$,  
$J$ denotes $(0, T]$ for some $T >0$, 
 $\gamma>0$ is the Gilbert damping parameter and $\beta$ is an exchange parameter. When  $\beta \neq 0$, it is often referred to as the Landau-Lifshitz-Gilbert equation.    The solution of \eqref{e_original model} preserves pointwisely its magnitude, i.e.,
  \begin{equation}\label{e_length preserving}
\aligned
\frac{d}{dt} | \textbf{m}(\textbf{x},t)  |^2=0,
\endaligned
\end{equation} 
which, together with  \eqref{e_initial condition},   implies that $ | \textbf{m}(\textbf{x},t )|=1 \;\;\forall x,\,t$.
It also satisfies the following energy dissipation law:
   \begin{equation}\label{e_energy dissipative law}
\aligned
\frac{d E}{dt}  = - \gamma \| \textbf{m} \times \Delta \textbf{m} \|^2,\quad\text{with }\; E= \frac{1}{2} \int_{\Omega} |\nabla \textbf{m} |^2 d \textbf{x}.
\endaligned
\end{equation}

The  Landau-Lifshitz equation  is a  fundamental equation in physics, and its accurate and efficient numerical simulation plays an important role in  understanding both the statics and dynamics in ferromagnetic materials  \cite{kruzik2006recent,lakshmanan1984landau,cheng2022length}, and has  attracted much attention in the past  decades, \cite{weinan2001numerical,alouges2006convergence,gao2014optimal,an2021optimal,akrivis2021higher,gui2022convergence} and the references therein. 

Since $ | \textbf{m}(\textbf{x},t ) | =1$ holds under the initial condition $|\textbf{m}_0(\textbf{x}) |=1$, it is important to construct numerical schemes  to satisfy  these point-wise constraints at the discrete level.
 There are essentially  three different strategies to deal with this constraint. (i) Penalty method \cite{pistella1999numerical,prohl2001computational}: introducing a suitable penalty term  in the original equation to relax the point-wise constraint. This approach has been widely used in the numerical simulation for liquid crystal flows \cite{badia2011finite,liu2000approximation}. Main drawbacks of this approach are: it does not enforce the length constraint exactly, and   it is challenging to numerically deal with the penalized formulation. 
   (ii) Projection method \cite{weinan2001numerical,bartels2008numerical}: projecting the intermediate magnetization $\tilde{\textbf{m}}^{n+1}$ onto the unit sphere $  \tilde{\textbf{m}}^{n+1} / |  \tilde{\textbf{m}}^{n+1} | $ after solving the equation at every time level. This approach  is very simple to implement, and has been frequently used \cite{gao2014optimal,kim2017mimetic,an2021optimal,an2022analysis,gui2022convergence}. However, it is difficult to enforce energy dissipation and to construct higher-order schemes.
 (iii) Lagrange multiplier method: introducing a Lagrange multiplier to enforce the length constraint  \cite{cheng2022length}. The simplest version of this approach is equivalent to a projection method, but their higher-order versions are different. This approach is easy to implement, and capable to  enforce energy dissipation  at the expense of solving one nonlinear algebraic equation. But its error analysis appears to be difficult. 

Error analysis of numerical schemes for the Landau-Lifshitz equation  \eqref{e_original model} is difficult due to its highly nonlinear nature and the length constraint. E and Wang established  a first-order error estimate  for the projection scheme in \cite{weinan2001numerical}.
An, Gao, and Sun \cite{an2021optimal} (see also \cite{chen2021convergence}) presented an error analysis for first- and second-order  semi-implicit projection finite difference schemes,  their analysis relies on inverse inequalities, and requires 
$h^2\le \Delta t\le h^{1+\epsilon_0} $\footnote{More recently  they improved in \cite{an2022analysis} the condition to $\Delta t = O(\epsilon_0 h)$ with some small $\epsilon_0$ for a first-order projection scheme with explicit treatment of $ \textbf{m} \times ( \textbf{m} \times \Delta \textbf{m} )$.} with  $\epsilon_0\in (0,1)$ if the nonlinear term $ \textbf{m} \times ( \textbf{m} \times \Delta \textbf{m} )$ is treated explicitly while this condition is not needed if this nonlinear term  is treated semi-explicitly. 
 On the other hand,  Gui, Li and Wang \cite{gui2022convergence} derived an optimal-order error estimate for a linearly implicit, lumped mass FEM on rectangular mesh under the condition  $ \Delta t \geq  \kappa h^{r+1} $ for some $ r >1$, where $\kappa$ is any positive constant. 

There are a number of equivalent formulations of  the Landau-Lifshitz equation. An interesting formulation is 
\begin{equation}\label{e_original model2}
	\alpha \partial_t  \textbf{m} + \textbf{m}\times  \partial_t  \textbf{m} =(I- \textbf{m} \textbf{m}^{ t })\Delta  \textbf{m},
\end{equation}
where $\alpha$ is a damping parameter which can be expressed in terms of $\beta, \,\gamma$. An advantage of this formulation is that its nonlinear term is easier to treat  in error analysis, but on the other hand, it is more complicated to implement numerically. Akrivis et al. \cite{akrivis2021higher} constructed linearly implicit  time discretizations up to order 5 combined with higher-order non-conforming finite elements in space, which satisfy a discrete energy dissipation low but do not ensure normalization of the magnetization.
 Alouges et al. \cite{alouges2008new,alouges2006convergence,alouges2014convergent} proposed the tangent plane schemes with second-order in time and provided a proof of convergence  toward a weak solution. We refer to \cite{akrivis2021higher,alouges2008new,alouges2006convergence,alouges2014convergent,alouges2012convergent} and the references therein for more details on numerical approximations of \eqref{e_original model2}. 

The main purpose of this paper is to construct a new  class of high-order IMEX schemes for the Landau-Lifshitz equation, with a generalized SAV (GSAV) approach \cite{huang2022new} to satisfy a (modified) energy dissipation law, and with a projection approach to preserve the pointwise length constraint, and to carry out a rigorous error analysis. Our main contributions are:
\begin{itemize}
\item The new schemes enjoy  the following advantages: {\color{black}(i) they are  purely linear, and at each time step only require solving  (a) decoupled  elliptic equations with constant coefficients  when the nonlinear term is treated fully explicitly, or (b) a coupled elliptic system with variable coefficients  when the nonlinear term is treated semi-implicitly;} (ii) they preserve the pointwise length constraint; (iii) they satisfy a (modified) energy dissipation law, and  their solutions are uniformly  bounded.
\item We carry out a rigorous error analysis for the semi-discretized schemes up to fifth-order in a unified framework and establish error estimates  in 
$l^{\infty}(0,T;H^1(\Omega)) \bigcap l^{2}(0,T;H^2(\Omega))$ for the  magnetization field under a mild condition on the  exchange parameter $\beta$ if  $\beta \ne 0$, and on   $\Delta t \leq (1+2^{q+1}C_0^{q})^{-1}$, where $C_0$ is a positive constant and $q$ is the order of numerical scheme (except that $q=2$ for the first-order scheme). 


\end{itemize}
To the best of the authors' knowledge, these are the first error estimates for higher-order numerical schemes which enforce normalization of the magnetization for the Landau-Lifshitz equation  \eqref{e_original model}.  Note that these estimates are established for the semi-discretized (in time) schemes, but it is expected that  error estimates for consistent fully discretized schemes  can be derived  without severe time step constraints.

The paper is organized as follows. In Section 2,  we present an equivalent  formulation of the Landau-Lifshitz equation and some preliminaries needed in the sequel. In Section 3, we construct a class of  length preserving  IMEX-GSAV schemes, and derive an unconditional bound for the numerical solution. In Section 4, we carry out a rigorous error analysis for the new schemes up to fifth-order in a unified framework.   \textcolor{black}{ In Section 5, we carry out a rigorous error analysis for the case where the  nonlinear term is treated semi-implicitly.}  We present some  numerical experiments   in Section 6 to  validate our theoretical results, and conclude with a few remarks in the final section.

  \section{Reformulations for the Landau-Lifshitz equation and preliminaries}
In this section, we first present an equivalent formulation of the Landau-Lifshitz equation, 
and then describe some notations and results which will be frequently used in this paper.

 The Landau-Lifshitz equation \eqref{e_original model} has several equivalent forms. Since it is very difficult to deal with the nonlinear term $\textbf{m} \times ( \textbf{m} \times \Delta \textbf{m} )$ implicitly while an explicit treatment will lead to a severe time step constraint, we first rewrite the Landau-Lifshitz equation \eqref{e_original model} as
   \begin{equation}\label{e_LLmodel_reformulation}
\aligned
 \frac{\partial \textbf{m}}{\partial t} =  -\beta \textbf{m} \times \Delta \textbf{m}+ \gamma \Delta\textbf{m} +\gamma |\nabla \textbf{m} |^2 \textbf{m} ,
     \quad &\  {\rm in} \ \Omega\times J,
\endaligned
\end{equation}  
by using $| \textbf{m} |=1$ and the identity
\begin{equation}\label{e_curl operator}
	\aligned
	\textbf{a} \times ( \textbf{b} \times \textbf{c} ) = (\textbf{a} \cdot \textbf{c}) \textbf{b} - (\textbf{a} \cdot \textbf{b}) \textbf{c}, \ \ \textbf{a},\textbf{b}, \textbf{c} \in \mathbb{R}^3.
	\endaligned
\end{equation}
A case of particular interest is when $\beta=0$, which leads to 
  \begin{equation}\label{e_LLmodel}
\aligned
 \frac{\partial \textbf{m}}{\partial t} = \gamma\Delta\textbf{m} + \gamma |\nabla \textbf{m} |^2 \textbf{m}\textcolor{black}{,}
     \quad &\ {\rm in}  \ \Omega\times J\textcolor{black}{.}
\endaligned
\end{equation} 
 The above equation is also referred to as the heat flow for harmonic maps, and has been well studied (cf.  \cite{lin2008analysis,gui2022convergence}).
    
We now describe  some notations and results which will be frequently used in this paper. Throughout the paper, we use $C$, with or without subscript, to denote a positive
constant, which could have different values at different appearances.

We  use the standard notations $L^2(\Omega)$, $H^k(\Omega)$ and $W^{k,p}(\Omega)$ to denote the usual Sobolev spaces over $\Omega$. The norm corresponding to $H^k(\Omega)$ will be denoted simply by $\|\cdot\|_k$. In particular, $\|\cdot\|$ and  $(\cdot,\cdot)$ are used to denote the norm and the inner product in $L^2(\Omega)$, respectively. The vectors and vector spaces will be indicated by boldface type.

 The following lemmas will be frequently used in the sequel:
 
 \begin{lemma}(H\"older inequality) \label{lem: Holder inequality}
 Let $p, q, s >0$ such that  $\frac{1}{p}+\frac{1}{q}+\frac{1}{s}=1$. Then for vector functions $\textbf{u} \in \textbf{L}^p(\Omega)$, $\textbf{v}\in \textbf{L}^q(\Omega)$, and scalar function $w \in L^s(\Omega)$,  we have
 \begin{equation}\label{e_Preliminaries1}
\aligned
 \int_{\Omega} | (\textbf{u}, \textbf{v}) w | d \textbf{x} \leq \| \textbf{u} \|_{\textbf{L}^p} \| \textbf{v} \|_{\textbf{L}^q} \| w \|_{L^s}.
\endaligned
\end{equation}
\end{lemma}
 
  \begin{lemma}(Interpolation inequalities)  \label{lem: Interpolation inequalities}
  	For any $\textbf{f}$, then there exists a positive
  		constant $C$ such that
 \begin{equation}\label{e_Preliminaries2}
\aligned
\| \textbf{f} \|_{\textbf{L}^k} \leq C \| \textbf{f} \|_{\textbf{L}^2}^{ \frac{6-k}{ 2k } } \| \textbf{f} \|_{\textbf{H}^1}^{ \frac{3k-6}{ 2k } }, \quad 3\le k\le 6,
\endaligned
\end{equation}
and
 \begin{equation}\label{e_Preliminaries3}
\aligned
\| \textbf{f} \|_{\textbf{L}^{\infty}} \leq C \| \textbf{f} \|_{\textbf{H}^1}^{ \frac{1}{ 2 } } \| \textbf{f} \|_{\textbf{H}^2}^{ \frac{1}{ 2 } }.
\endaligned
\end{equation}
\end{lemma}

We will frequently use the following discrete version of the Gronwall lemma (see, for instance, \cite{shen1990long,HeSu07}):

\medskip
\begin{lemma} \label{lem: gronwall2}
Let $a_k$, $b_k$, $c_k$, $d_k$, $\gamma_k$, $\Delta t_k$ be non negative real numbers such that
\begin{equation}\label{e_Gronwall3}
\aligned
a_{k+1}-a_k+b_{k+1}\Delta t_{k+1}+c_{k+1}\Delta t_{k+1}-c_k\Delta t_k\leq a_kd_k\Delta t_k+\gamma_{k+1}\Delta t_{k+1}
\endaligned
\end{equation}
for all $0\leq k\leq m$. Then
 \begin{equation}\label{e_Gronwall4}
\aligned
a_{m+1}+\sum_{k=0}^{m+1}b_k\Delta t_k \leq \exp \left(\sum_{k=0}^md_k\Delta t_k \right)\{a_0+(b_0+c_0)\Delta t_0+\sum_{k=1}^{m+1}\gamma_k\Delta t_k \}.
\endaligned
\end{equation}
\end{lemma}
  
  \section{Norm preserving and energy decreasing IMEX-GSAV schemes}
 We construct  in this section higher-order norm preserving and energy decreasing schemes for the  Landau-Lifshitz equation \eqref{e_LLmodel_reformulation} based on the IMEX BDF-$l$ formulae.

  Set $$\Delta t=T/N,\ t^n=n\Delta t, \ d_t g^{n+1}=\frac{g^{n+1}-g^n}{\Delta t},
\ {\rm for} \ n\leq N,$$
 and introduce an SAV
 \begin{equation}\label{e_definition of R}
\aligned
R(t)=E(\textbf{m})+K_0, \ E(\textbf{m}) = \frac1 2 \| \nabla \textbf{m} \|^2,
\endaligned
\end{equation}
where $K_0$ is a positive constant. We  expand the governing system as follows:
  \begin{numcases}{}
  \frac{\partial \textbf{m} }{\partial t} = \gamma \Delta\textbf{m} + \gamma |\nabla \textbf{m} |^2 \textbf{m} - \beta \textbf{m}\times  \Delta\textbf{m},  \label{e_model_transform1}\\
  \frac{ d R}{ d t} = - \frac{R}{ E( \textbf{m} ) +K_0} \gamma \| \textbf{m} \times \Delta \textbf{m} \|^2, \label{e_model_transform2} \\
  | \textbf{m} (x,t)|=1\quad\forall x,\, t. \label{e_model_transform3}
\end{numcases}
 
It is clear that with $R(0)=E(\textbf{m})|_{t=0}+K_0$, the solution of the original system \eqref{e_LLmodel} is also a solution of the above system. We construct below IMEX GSAV schemes for the expanded system \eqref{e_model_transform1}\textcolor{black}{-}\eqref{e_model_transform3}.

Assuming $\textbf{m}^j$, $\tilde{\textbf{m}}^{j}$ with $j=n,n-1,\ldots, n-l+1$ are given, we solve $\tilde{\textbf{m}}^{n+1}$ from 
\begin{equation}\label{e_High-order1}
\aligned
D_l \tilde{\textbf{m}}^{n+1} = & \gamma \Delta \tilde{\textbf{m}}^{n+1}+ S \Delta( \tilde{\textbf{m}}^{n+1}- B_l( \textbf{m}^{n}))\\
&+ \gamma  | \nabla B_l( \textbf{m}^{n})  |^2 B_l( \textbf{m}^{n} )- \beta B_l( \textbf{m}^{n} ) \times \Delta \textcolor{black}{ B_l ( \tilde{\textbf{m}}^{n} ) },
\endaligned
\end{equation} 
where $S\ge 0$ is a stabilization parameter, $D_l$ denotes the BDF-$l$ formula, and $B_l$ is the $l$-th order Adams-Bashforth extrapolation operator. 

Next we compute $R^{n+1}$,  $\xi^{n+1}$  from 
\begin{equation}\label{e_High-order2}
	\aligned
	\frac{ R^{n+1}  - R^{n} } { \Delta t } = & - \gamma \xi^{n+1} \| B_l(\textbf{m}^{n} ) \times \Delta \tilde{\textbf{m}}^{n+1} \|^2, \ \ 
	\xi^{n+1}= \frac{ R^{n+1} }{ E( \tilde{\textbf{m}}^{n+1} )+K_0 },
	\endaligned
\end{equation}  
where  $ R(t^{n+1}) = E( \textbf{m}(t^{n+1} )) +K_0 $ and $ E( \textbf{m}(t^{n+1} )) = \frac{1}{2} \int_{\Omega} | \nabla \textbf{m}(t^{n+1} ) |^2 d \textbf{x}$.

Finally we update $\textbf{m}^{n+1}$ by 
\begin{equation}\label{e_High-order_correct}
	\aligned
	\textbf{m}^{n+1} = \frac{ \hat{\textbf{m}}^{n+1} } { | \hat{\textbf{m}}^{n+1}  | } , \ \text{with }\ \hat{\textbf{m}}^{n+1} = \eta_{l}^{n+1} \tilde{\textbf{m}}^{n+1}+ (\eta_{l}^{n+1} -1)^w, 
	\endaligned
\end{equation} 
where 
\begin{equation}\label{e_High-order_eta}
	\aligned
	\eta_{1}^{n+1} = 1- (1-\xi^{n+1} )^2; \  \eta_{l}^{n+1} = 1- (1-\xi^{n+1} )^l,\;l \ge 2.
	\endaligned
\end{equation} 
 Here $w$ can be chosen as any positive integer. In general $w=1$ is sufficient, and our numerical  experiments suggest that $w \geq 2$ may lead to better results. For readers' convenience, $D_l $ and $B_l$ with  $l=1,2,3$ are given below. Formulae for other $l$ can be easily derived from Taylor expansions:

first-order scheme: 
\begin{equation*}\label{e_First-order_operator}
\aligned
D_1 \tilde{\textbf{m}}^{n+1} = \frac{ \tilde{\textbf{m}}^{n+1} - \tilde{\textbf{m}}^{n} }{ \Delta t } , \  B_1( \textbf{m}^{n}) =  \textbf{m}^{n} ;
\endaligned
\end{equation*}

second-order scheme: 
\begin{equation*}\label{e_Second-order_operator}
\aligned
D_2 \tilde{\textbf{m}}^{n+1} = \frac{ 3 \tilde{\textbf{m}}^{n+1} - 4 \tilde{\textbf{m}}^{n} + \tilde{\textbf{m}}^{n-1} }{ 2 \Delta t }, \ 
B_2( \textbf{m}^{n}) = 2 \textbf{m}^{n} -  \textbf{m}^{n-1};
\endaligned
\end{equation*} 

third-order scheme: 
\begin{equation*}\label{e_Third-order_operator}
\aligned
D_3 \tilde{\textbf{m}}^{n+1} = \frac{ 11 \tilde{\textbf{m}}^{n+1} - 18 \tilde{\textbf{m}}^{n} + 9 \tilde{\textbf{m}}^{n-1} - 2 \tilde{\textbf{m}}^{n-2} }{ 6 \Delta t }, \
B_3( \textbf{m}^{n} ) = 3 \textbf{m}^{n} -  3 \textbf{m}^{n-1} + \textbf{m}^{n-2}.
\endaligned
\end{equation*} 

It is clear that the above scheme admits a unique solution, and it is very easy to implement as it only requires solving a vector Poisson type equation in \eqref{e_High-order1}.

\begin{rem} \label{special discretization}
In this work we only consider the time discretization. Note that the proposed IMEX GSAV schemes \eqref{e_High-order1}-\eqref{e_High-order_eta} can be easily implemented with finite difference and spectral methods, but theirs implementation with finite element methods needs special consideration when $\beta\ne 0$ due the explicit treatment for the term $\beta B_l( \textbf{m}^{n} ) \times \Delta B_l( \textbf{m}^{n} )$. An additional, well-known difficulty in the case of finite element methods is that the normalization of a non-constant polynomial vector field is not polynomial. One can combine  \eqref{e_High-order1}-\eqref{e_High-order_eta} with a Gauss-Seidel techniques in \cite{wang2001gauss} to treat the term  $\beta B_l( \textbf{m}^{n} ) \times \Delta B_l( \textbf{m}^{n} )$ semi-implicitly when the parameter $\beta$ is relatively large.
\end{rem}
 \medskip

 \begin{theorem}\label{thm_energy stability_high order} Let $\{\textbf{m}^{n},\tilde {\textbf{m}}^{n},\hat{\textbf{m}}^{n}\}$ be the solution of  \eqref{e_High-order1}-\eqref{e_High-order_eta}.
Given $R^n>0$,  we have  $\xi^{n+1}>0$ and 
 \begin{equation}\label{e_stability1}
\aligned
0<R^{n+1} <  R^{n}, \ \ \forall n \leq T/\Delta t.
\endaligned
\end{equation} 
In addition, there exists a constant $M_T$ independent of $\Delta t$ such that
\begin{equation}\label{e_stability2}
\aligned
\| \nabla \hat{\textbf{m}}^{n+1}  \| + \sum_{k=0}^{n} \Delta t \gamma \xi^{k+1} \|  B_l(\textbf{m}^{k} ) \times \Delta \tilde{\textbf{m}}^{k+1} \|^2 
\leq M_T, \ \ \forall n \leq T/\Delta t, and\ n=l-1,\cdots,\ N-1.
\endaligned
\end{equation} 

\end{theorem}

\begin{proof}
Given $R^n>0$, it follows from \eqref{e_High-order2} that 
\begin{equation}\label{e_High-order_Stability1}
\aligned
R^{n+1} =\frac{1 }{ 1+ \Delta t  \gamma \frac{ \| B_l(\textbf{m}^{n} ) \times \Delta \tilde{\textbf{m}}^{n+1} \|^2 }{E( \tilde{\textbf{m}}^{n+1} )+K_0 } }R^n > 0.
\endaligned
\end{equation} 
In addition, multiplying \eqref{e_High-order2} with $\Delta t$ and taking the sum for $n$ from 0 to $m$ result in
\begin{equation}\label{e_High-order_Stability2}
\aligned
R^{m+1} + \sum_{n=0}^{m} \Delta t  \gamma \xi^{n+1} \|   B_l(\textbf{m}^{n} ) \times \Delta \tilde{\textbf{m}}^{n+1} \|^2 = R^0.
\endaligned
\end{equation} 
Denote $R^0 := M;$ the \eqref{e_High-order_Stability2} implies 
\begin{equation}\label{e_High-order_Stability3}
\aligned
R^{m+1} + \sum_{n=0}^{m} \Delta t  \gamma \xi^{n+1} \|   B_l(\textbf{m}^{n} ) \times \Delta \tilde{\textbf{m}}^{n+1} \|^2 \leq M.
\endaligned
\end{equation}
It then follows from \eqref{e_High-order2} that 
\begin{equation}\label{e_High-order_Stability4}
\aligned
 \xi^{n+1} = \frac{ R^{n+1} }{ E( \tilde{\textbf{m}}^{n+1} )+K_0} \leq \frac{ 2M }{ \| \nabla \tilde{\textbf{m}}^{n+1} \|^2 +1 },
\endaligned
\end{equation} 
where, without loss of generality, we assume that the positive constant $K_0 \geq \frac 1 2$. We derive from \eqref{e_High-order_eta} that $\eta_l^{n+1}  =  \xi^{n+1} P_{q}( \xi^{n+1} )$, where $ P_{q} $ is a polynomial function of degree $q$ with $q=1$ for $l=1$ and $q=l-1$ for $l>1$. We can then derive from \eqref{e_High-order_Stability4} that there exists a positive constant $M_1$ such that
\begin{equation}\label{e_High-order_Stability5}
\aligned
| \eta_l^{n+1} | = | \xi^{n+1} P_{q}( \xi^{n+1} ) | \leq \frac{ M_1 }{ \| \nabla \tilde{\textbf{m}}^{n+1} \|^2 +1 } .
\endaligned
\end{equation} 
 Thus we have 
\begin{equation}\label{e_High-order_Stability6}
\aligned
\| \nabla \hat{\textbf{m}}^{n+1}  \|^2 = (  \eta_l^{n+1} )^2  \| \nabla \tilde{\textbf{m}}^{n+1}  \|^2 \leq \left( \frac{ M_1 }{ \| \nabla \tilde{\textbf{m}}^{n+1} \|^2 +1 }  \right) ^2  \| \nabla \tilde{\textbf{m}}^{n+1}  \|^2 \leq M_1^2,
\endaligned
\end{equation} 
which implies the desired results \eqref{e_stability2}.
\end{proof}
\begin{rem}
	The bound for $B_l(\textbf{m}^{k} ) \times \Delta \tilde{\textbf{m}}^{k+1}$  in \eqref{e_stability2} is particularly useful in the error analysis.
\end{rem}



 \section{Error estimate}
In this section, we carry out a  rigorous error analysis for  \eqref{e_High-order1}-\eqref{e_High-order_eta} with $\beta \ne 0$. For the sake of simplicity we  set the stabilization parameter $S=0$.

We  denote
   \begin{numcases}{}
\displaystyle \tilde{e}_{\textbf{m}}^{n+1}= \textbf{m}(t^{n+1}) - \tilde{\textbf{m}}^{n+1},\ \
\displaystyle e_{\textbf{m}}^{n+1}= \textbf{m}(t^{n+1}) - \textbf{m}^{n+1}, \notag\\
\displaystyle e_{R }^{n+1}= R(t^{n+1}) - R^{n+1} .
\end{numcases}

We first recall the following  important result by Nevanlinna and Odeh \cite{nevanlinna1981multiplier} based on  Dahlquist's G-stability theory.
\begin{lemma} \label{lem: multiplier step}
\cite{nevanlinna1981multiplier} For $1\leq l \leq 5$, there exists $0\leq \tau_l<1$, a positive definite symmetric matrix $G=(g_{i,j}) \in \mathbb{R}^{l,l} $, and real numbers $ \xi_0,\ldots, \xi_5$ such that
\begin{equation}\label{e_multiplier step}
\aligned
\Delta t ( D_l \textbf{m}^{n+1} , \textbf{m}^{n+1} - \tau_l \textbf{m}^{n} ) = & \sum\limits_{i,j=1}^l g_{i,j}
( \textbf{m}^{n+1+i-l } ,  \textbf{m}^{n+1+j-l } ) \\
& - \sum\limits_{i,j=1}^l g_{i,j}
( \textbf{m}^{n+i-l } ,  \textbf{m}^{n+j-l } )  + \|  \sum\limits_{i=0}^l  \xi_i \textbf{m}^{n+1+i-l} \|^2,
\endaligned
\end{equation} 
where the smallest possible values of $\tau_l$ are
$$ \tau_1=\tau_2=0, \ \tau_3=0.0836, \ \tau_4 =0.2878, \ \tau_5 =0.8160. $$
\end{lemma}

The proof of the following main result relies essentially  on the stability result in \eqref{e_stability2} and the above lemma.
\medskip
\begin{theorem}\label{the: error_estimate_final}
 \textcolor{black}{ 
 Supposing that the damping parameter $\gamma$ satisfies
 \begin{equation}\label{e_beta_value1}
\aligned
 \gamma > \frac{  ( 1+ \tau_l + l \tau_l ) }{ 2(1-\tau_l )}  | \beta | ,
\endaligned
\end{equation} 
}
 and assuming  $\textbf{m}\in H^{l+1}(0,T;\textbf{L}^2(\Omega))\bigcap H^{l}(0,T;\textbf{W}^{1,3}(\Omega)) \bigcap L^{\infty}(0,T; H^4(\Omega))$, then for the scheme \eqref{e_High-order1}-\eqref{e_High-order_eta}, we have
 \begin{equation*}
\aligned
& \|  \textbf{m}^{n+1} - \textbf{m}(t^{n+1}) \|^2 + \|  \nabla( \textbf{m}^{n+1} - \textbf{m}(t^{n+1}) ) \|^2 + \|  \hat{\textbf{m}}^{n+1} - \textbf{m}(t^{n+1}) \|^2 \\
& + \|  \nabla( \hat{\textbf{m}}^{n+1} - \textbf{m}(t^{n+1}) ) \|^2 + \|  \tilde{\textbf{m}}^{n+1} - \textbf{m}(t^{n+1}) \|^2 \\
& + \|  \nabla( \tilde{\textbf{m}}^{n+1} - \textbf{m}(t^{n+1}) ) \|^2 +  \Delta t \sum\limits_{k=0}^{n} \|  \Delta( \tilde{\textbf{m}}^{k+1} - \textbf{m}(t^{k+1}) ) \|^2 \\
\leq & C(\Delta t)^{2l}, \ \ \ \forall \ 0\leq n\leq N-1,\quad 1\le l\le 5,
\endaligned
\end{equation*} 
where $C$ is a positive constant  independent of $\Delta t$.
\end{theorem}

\begin{proof}
We shall first prove, by an induction process with a bootstrap argument, that there exists a positive constant $C_0$  such that
\begin{equation}\label{e_error1}
\aligned
| 1- \xi^k | \leq C_0 \Delta t, \ \forall k \leq T/ \Delta t,
\endaligned
\end{equation}
\begin{equation}\label{e_error2}
\aligned
\| \tilde{e}_{\textbf{m}}^{k}  \|_{H^2}  \leq (\Delta t)^{1/4} , \ \forall k \leq T/ \Delta t.
\endaligned
\end{equation}

Obviously \eqref{e_error1} and \eqref{e_error2} hold for $k=0$. Now we suppose 
\begin{equation}\label{e_error3}
\aligned
| 1- \xi^k | \leq C_0 \Delta t, \ \forall k \leq n,
\endaligned
\end{equation}
\begin{equation}\label{e_error4}
\aligned
\| \tilde{e}_{\textbf{m}}^{k}  \|_{H^2} \leq (\Delta t)^{1/4} , \ \forall k \leq n,
\endaligned
\end{equation}
and  prove below that  $| 1- \xi^{n+1} | \leq C_0 \Delta t$ and $ \| \tilde{e}_{\textbf{m}}^{n+1}  \|_{H^2}  \leq (\Delta t)^{ 1/4}  $ hold true. We shall carry out the induction proof in three steps below.

\noindent{\bf Step 1:  A $H^2$ bound for $\tilde{e}_{\textbf{m}}^{n+1}$.} 
Following a similar procedure as in \cite{huang2021stability} and recalling the inequality
\begin{equation}\label{e_error_akbk}
\aligned
( a+b )^l \leq 2^{l-1} (a^l +b^l),\  \forall a,b>0, l\geq 1,
\endaligned
\end{equation} 
we can easily obtain from the induction assumption  that  under the condition $\Delta t \leq \min\{ \frac{1}{2^{l+2} C_0^l },1\}$, we have
\begin{equation}\label{e_error5}
\aligned
\frac{1}{2} \leq | \xi^k |, \  | \eta^k_l |  \leq 2, \quad\forall k\le n.
\endaligned
\end{equation}

 We first derive a useful relation between the error functions $e_{\textbf{m}}^{k} $,  $ \hat{e}_{\textbf{m}}^{k} $ and $ \tilde{e}_{\textbf{m}}^{k} $  in order to bound the nonlinear terms in the following analysis.
\begin{lemma}\label{lem: relation for error functions}
 Assuming that \eqref{e_error3} and \eqref{e_error4} hold for all $1 \leq k \leq n$, $2 \leq p \leq 4$ and $1 \leq l \leq 5$, then we have 
\begin{equation}\label{e_relation_L2}
\aligned
& \| e_{\textbf{m}}^{k} \|_{W^{1,p}} \leq C \| \hat{e}_{\textbf{m}}^{k} \|_{W^{1,p}} \leq C \| \tilde{e}_{\textbf{m}}^{k} \|_{W^{1,p}} + C (C_0 \Delta t)^q, \quad 1 \leq k \leq n,
\endaligned
\end{equation}
where $C$ is a positive constant  independent of $\Delta t$ and $C_0$, and $q=2$ when $l=1$, and $q=l$ when $l=2,3,4,5$ respectively.
\end{lemma}

\begin{proof}
Assuming that \eqref{e_error4} holds for all $1 \leq k \leq n$, we have
\begin{equation}\label{e_lem_relation1}
\aligned
& \| \tilde{e}_{\textbf{m}}^{k} \|_{ L^{\infty} } \leq C \| \tilde{e}_{\textbf{m}}^{k} \|_{H^1}^{1/2}  \| \tilde{e}_{\textbf{m}}^{k} \|_{H^2}^{1/2} \leq C_1 (\Delta t)^{ l-5/8 } \leq \frac{1}{8 }, 
\endaligned
\end{equation}
under the condition  $\Delta t \leq  (\frac{1 }{ 8 C_1} )^{4 }$. By using \eqref{e_High-order_eta}, \eqref{e_error3} and \eqref{e_error4}, we have 
\begin{equation}\label{e_lem_relation2}
\aligned
 \| \hat{e}_{\textbf{m}}^{k} \|_{ L^{\infty} } \leq & \| \textbf{m}(t^{k}) -   \tilde{\textbf{m}}^{k} \| _{ L^{\infty} }+ \| \tilde{\textbf{m}}^{k} - \hat{\textbf{m}}^{k} \|_{ L^{\infty} }
\leq  \| \tilde{e}_{\textbf{m}}^{k} \|_{ L^{\infty} } + | 1- \eta^k_l  | ( \| \tilde{\textbf{m}}^{k} \|_{ L^{\infty} } + | 1- \eta^k_l |^{w-1} ) \\
\leq &  \| \tilde{e}_{\textbf{m}}^{k} \| _{ L^{\infty} }+ | 1- \xi^k |^q ( \| \tilde{\textbf{m}}^{k} \|_{ L^{\infty} } + | 1- \eta^k_l |^{w-1}  ) \\
 \leq & \frac{1}{8} + ( \| \tilde{\textbf{m}}^{k} \|_{ L^{\infty} } + | 1- \eta^k_l |^{w-1} )  ( C_0  \Delta t )^q \leq \frac{1}{4},
 \endaligned
\end{equation}
where $\Delta t \leq \frac{1}{  17^{1/q} C_0 }$. Here $q=2$ when $l=1$, and $q=l$ when $l=2,3,4,5$.
Then point-wise in $\Omega$, we have 
\begin{equation}\label{e_lem_relation3}
\aligned
 \frac{1}{2} \leq | \hat{\textbf{m}}^{k} | \leq \frac{3}{2} . 
 \endaligned
\end{equation}
By using exactly the same procedure as \cite{gui2022convergence,akrivis2021higher}, there holds
\begin{equation}\label{e_lem_relation4}
\aligned
| \frac{f}{ |f| } -  \frac{g }{ |g | } | = | \frac{ f ( |g| -|f |) -|f |( g -f ) }{ |f| |g|} | \leq 2 \frac{ | f-g |}{ |g |} \leq C | f-g |,
 \endaligned
\end{equation}
where the constant $C$ depends on $|g|$, and similarly we have
\begin{equation}\label{e_lem_relation5}
\aligned
| \nabla \frac{f}{ |f| } -  \nabla \frac{g }{ |g | } | \leq C |\nabla g| |f-g| +C | \nabla (f-g)|. 
 \endaligned
\end{equation}
Thus we have
\begin{equation}\label{e_lem_relation6}
\aligned
\| e_{\textbf{m}}^{k} \|_{W^{1,p}} =  \| \textbf{m}(t^{k}) -   \textbf{m}^{k} \|_{W^{1,p}} =  \| \frac{ \textbf{m}(t^{k}) } { | \textbf{m}(t^{k}) | } -  \frac{ \hat{\textbf{m}}^{k} }{ |\hat{\textbf{m}}^{k} |} \| _{W^{1,p}} \leq C\| \hat{e}_{\textbf{m}}^{k} \|_{W^{1,p}} .
 \endaligned
\end{equation}
Similarly as \eqref{e_lem_relation2} and using \eqref{e_error4}, we can obtain
\begin{equation*}\label{e_lem_relation7}
\aligned
 \| \hat{e}_{\textbf{m}}^{k} \|_{W^{1,p}} \leq & \| \textbf{m}(t^{k}) -   \tilde{\textbf{m}}^{k} \|_{W^{1,p}} + \| \tilde{\textbf{m}}^{k} - \hat{\textbf{m}}^{k} \|_{W^{1,p}} \\
\leq &  \| \tilde{e}_{\textbf{m}}^{k} \|_{W^{1,p}} + | 1- \eta^k_l | ( \| \tilde{\textbf{m}}^{k} \|_{W^{1,p}} +  | 1- \eta^k_l |^{w-1} ) \\
\leq &  \| \tilde{e}_{\textbf{m}}^{k} \|_{W^{1,p}}+ | 1- \xi^k |^q  ( \| \tilde{\textbf{m}}^{k} \|_{W^{1,p}} +  | 1- \eta^k_l |^{w-1} ) \\
 \leq & \| \tilde{e}_{\textbf{m}}^{k} \| _{W^{1,p}} + C (C_0 \Delta t)^q.
 \endaligned
\end{equation*}
The proof is complete.
\end{proof}
 
We now continue with the proof of Theorem \ref{the: error_estimate_final}. Let $\textbf{R}_{\textbf{m}, l }^{n+1}$ be the truncation error defined by
\begin{equation}\label{e_error6}
\aligned
\textbf{R}_{\textbf{m},l }^{k+1}=\frac{\partial \textbf{m}(t^{k+1})}{\partial t} - D_l \textbf{m}(t^{k+1}) = \frac{1}{\Delta t} \sum\limits_{i=1}^l \delta_i \int_{t^{k+1-i}}^{t^{k+1}}(t^{k+1-i}-t)^l  \frac{\partial^{l+1} \textbf{m}}{\partial t^{l+1} }dt,
\endaligned
\end{equation}
with $ \delta_i $ being some fixed and bounded constants determined by the truncation errors.

Subtracting \eqref{e_High-order1}  from \eqref{e_model_transform1} at $t^{k+1}$, we obtain an error equation  corresponding to \eqref{e_High-order1}:
\begin{equation}\label{e_error7}
\aligned
 D_l  \tilde{e}_{\textbf{m}}^{k+1} -  \gamma \Delta \tilde{e}_{\textbf{m}}^{k+1}
= & -  \gamma \left( |\nabla B_l( \textbf{m}^{k} ) |^2 B_l( \textbf{m}^{k} ) -   |\nabla \textbf{m}( t^{k+1} ) |^2 \textbf{m}( t^{k+1} ) \right) - \textbf{R}_{\textbf{m},l }^{k+1} \\
& - \beta \left(  \textbf{m} (t^{k+1}) \times  \Delta\textbf{m} (t^{k+1}) - B_l( \textbf{m}^{k} )  \times  \Delta \tilde{ \textbf{m}}^{k+1} \right).
\endaligned
\end{equation}

Taking the inner product of \eqref{e_error7} with $ - \Delta t ( \Delta  \tilde{e}_{\textbf{m}}^{k+1} - \tau_l \Delta  \tilde{e}_{\textbf{m}}^{k} )$ and using Lemma \ref{lem: multiplier step}, we obtain
\begin{align}
& \sum\limits_{i,j=1}^l g_{i,j}
( \nabla \tilde{e}_{\textbf{m}}^{k+1+i-l } ,  \nabla \tilde{e}_{\textbf{m}}^{k+1+j-l } ) - \sum\limits_{i,j=1}^l g_{i,j}
( \nabla \tilde{e}_{\textbf{m}}^{k+i-l } ,  \nabla \tilde{e}_{\textbf{m}}^{k+j-l } ) \notag \\
& + \|  \sum\limits_{i=0}^l  \xi_i \nabla \tilde{e}_{\textbf{m}}^{k+1+i-l} \|^2 +  \Delta t  \gamma \| \Delta \tilde{e}_{\textbf{m}}^{k+1} \|^2 \label{e_error8}  \\
= &  \Delta t \gamma \left(  |\nabla B_l( \textbf{m}^{k} ) |^2 -   |\nabla B_l( \textbf{m}( t^{k} ) ) |^2 ,  B_l( \textbf{m}^{k} ) ( \Delta  \tilde{e}_{\textbf{m}}^{k+1} - \tau_l \Delta  \tilde{e}_{\textbf{m}}^{k} ) \right) \notag \\
& -  \Delta t \gamma \left(  |\nabla B_l( \textbf{m}( t^{k} ) ) |^2 B_l ( e_{\textbf{m}}^{k} ),   \Delta  \tilde{e}_{\textbf{m}}^{k+1} -  \tau_l \Delta  \tilde{e}_{\textbf{m}}^{k} \right) \notag \\
&+  \Delta t \gamma \left(  |\nabla B_l( \textbf{m}( t^{k} ) ) |^2 B_l( \textbf{m}( t^{k} ) ) -  |\nabla \textbf{m}( t^{k+1} ) |^2 \textbf{m}( t^{k+1} ) ,  \Delta  \tilde{e}_{\textbf{m}}^{k+1}  - \tau_l \Delta  \tilde{e}_{\textbf{m}}^{k} ) \right) \notag \\
& +  \Delta t \gamma (\Delta \tilde{e}_{\textbf{m}}^{k+1},  \tau_l \Delta  \tilde{e}_{\textbf{m}}^{k} )  +  \Delta t ( \textbf{R}_{\textbf{m},l}^{k+1}, \Delta  \tilde{e}_{\textbf{m}}^{k+1}  - \tau_l \Delta  \tilde{e}_{\textbf{m}}^{k} ) )  \notag \\
& + \textcolor{black}{ \Delta t \beta \left( \left(  \textbf{m} (t^{k+1}) \times  \Delta\textbf{m} (t^{k+1}) - B_l( \textbf{m}^{k} )  \times  \Delta B_l ( \tilde{\textbf{m}}^{k} ) \right), 
 \Delta  \tilde{e}_{\textbf{m}}^{k+1} - \tau_l \Delta  \tilde{e}_{\textbf{m}}^{k} \right) } \notag \\
= &\Delta t( I_1+ I_2 + I_3 + I_4+I_5 + I_6 ).  \notag
\end{align}
We bound the above six terms as follows. 
Using the H\"older inequality in Lemma \ref{lem: Holder inequality} and Lemma \ref{lem: relation for error functions}, we have 
\begin{align}
 | I_1 |  = & \gamma \left(  |\nabla B_l( \textbf{m}^{k} ) |^2 -   |\nabla B_l( \textbf{m}( t^{k} ) ) |^2 ,  B_l( \textbf{m}^{k} ) ( \Delta  \tilde{e}_{\textbf{m}}^{k+1} - \tau_l \Delta  \tilde{e}_{\textbf{m}}^{k} ) \right) \notag \\
= &  \gamma \left( \nabla B_l( e_{\textbf{m}}^{k} ) \nabla B_l( e_{\textbf{m}}^{k} )  , B_l( \textbf{m}^{k} ) ( \Delta  \tilde{e}_{\textbf{m}}^{k+1} - \tau_l \Delta  \tilde{e}_{\textbf{m}}^{k} )  \right) \notag  \\
& - 2 \gamma \left(  \nabla B_l( \textbf{m}( t^{k} ) ) \nabla B_l( e_{\textbf{m}}^{k} ) ,  B_l( \textbf{m}^{k} ) ( \Delta  \tilde{e}_{\textbf{m}}^{k+1} 
- \tau_l \Delta  \tilde{e}_{\textbf{m}}^{k} ) \right) \notag  \\
\leq & \gamma \| \nabla B_l( e_{\textbf{m}}^{k} ) \|_{L^4}  \| \nabla B_l( e_{\textbf{m}}^{k} ) \|_{L^4} \| B_l( \textbf{m}^{k} ) \Delta  \tilde{e}_{\textbf{m}}^{k+1} - \tau_l \Delta  \tilde{e}_{\textbf{m}}^{k} )  \|_{L^2} \label{e_error9_A} \\
& + 2  \gamma \| \nabla B_l( \textbf{m}( t^{k} ) ) \|_{L^6} \| \nabla B_l( e_{\textbf{m}}^{k}  ) \|_{L^3}  \| B_l( \textbf{m}^{k} ) \Delta  \tilde{e}_{\textbf{m}}^{k+1} - \tau_l \Delta  \tilde{e}_{\textbf{m}}^{k} )  \|_{L^2} \notag  \\
\leq &C \left( \|  \nabla B_l( \tilde{e}_{\textbf{m}}^{k} )  \|_{L^4} +  C (C_0 \Delta t)^q \right)^2 \| B_l( \textbf{m}^{k} ) \Delta  \tilde{e}_{\textbf{m}}^{k+1} - \tau_l \Delta  \tilde{e}_{\textbf{m}}^{k} )  \|_{L^2} \notag  \\
& + 2 C \| \nabla B_l( \textbf{m}( t^{k} ) ) \|_{L^6} ( \| \nabla B_l( \tilde{e}_{\textbf{m}}^{k} )  \|_{L^3} +  C (C_0 \Delta t)^q ) \| B_l( \textbf{m}^{k} ) \Delta  \tilde{e}_{\textbf{m}}^{k+1} - \tau_l \Delta  \tilde{e}_{\textbf{m}}^{k} )  \|_{L^2}. \notag 
\end{align}
Applying the interpolation inequality \eqref{e_Preliminaries2} and Cauchy-Schwarz inequality, the above term can be estimated by
\begin{align}
 | I_1 |  
\leq & C \| \nabla B_l( \tilde{e}_{\textbf{m}}^{k} )  \|_{L^2}^{1/2}   \| \Delta B_l( \tilde{e}_{\textbf{m}}^{k} ) \|_{L^2}^{3/2} \| B_l( \textbf{m}^{k})(  \Delta  \tilde{e}_{\textbf{m}}^{k+1} - \tau_l \Delta  \tilde{e}_{\textbf{m}}^{k} ) \|_{L^2} \notag \\
&+ C  \| \nabla B_l( \textbf{m}( t^{k} ) ) \|_{L^6} \| \nabla B_l( \tilde{e}_{\textbf{m}}^{k} ) \|_{L^2}^{1/2}   \| \Delta B_l( \tilde{e}_{\textbf{m}}^{k} )\|_{L^2}^{1/2}  \| B_l( \textbf{m}^{k})(  \Delta  \tilde{e}_{\textbf{m}}^{k+1} - \tau_l \Delta  \tilde{e}_{\textbf{m}}^{k} ) \|_{L^2} \notag  \\
& + C  ( C_0 \Delta t )^q  ( ( C_0 \Delta t )^q+ \| \nabla B_l( \textbf{m}( t^{k} ) ) \|_{L^6} + \|  \nabla B_l( \tilde{e}_{\textbf{m}}^{k} )  \|_{L^4} ) \| B_l( \textbf{m}^{k})(  \Delta  \tilde{e}_{\textbf{m}}^{k+1} - \tau_l \Delta  \tilde{e}_{\textbf{m}}^{k} ) \|_{L^2} \notag  \\
\leq & C  ( \Delta t )^{1/4} \| \nabla B_l( \tilde{e}_{\textbf{m}}^{k} ) \|_{L^2}^{1/2}   \| \Delta B_l( \tilde{e}_{\textbf{m}}^{k} ) \|_{L^2}^{1/2} \|  B_l( \textbf{m}^{k})(  \Delta  \tilde{e}_{\textbf{m}}^{k+1} - \tau_l \Delta  \tilde{e}_{\textbf{m}}^{k} ) \|_{L^2} \label{e_error9} \\
& + C \| \nabla B_l( \textbf{m}( t^{k} ) ) \|_{L^6} \| \nabla B_l( \tilde{e}_{\textbf{m}}^{k} ) \|_{L^2}^{1/2}   \| \Delta B_l( \tilde{e}_{\textbf{m}}^{k} ) \|_{L^2}^{1/2}  \| B_l( \textbf{m}^{k})(  \Delta  \tilde{e}_{\textbf{m}}^{k+1} - \tau_l \Delta  \tilde{e}_{\textbf{m}}^{k} ) \|_{L^2} \notag  \\
& + C  ( C_0 \Delta t )^q  ( ( C_0 \Delta t )^q+ \| \nabla B_l( \textbf{m}( t^{k} ) ) \|_{L^6} + \|  \nabla B_l( \tilde{e}_{\textbf{m}}^{k} )  ) \| B_l( \textbf{m}^{k})(  \Delta  \tilde{e}_{\textbf{m}}^{k+1} - \tau_l \Delta  \tilde{e}_{\textbf{m}}^{k} ) \|_{L^2}  \notag  \\
\leq &  \epsilon  \|  \Delta  \tilde{e}_{\textbf{m}}^{k+1} \|_{L^2}^2 +  \frac{ \epsilon   }{ l }   \sum\limits_{i=0}^{l-1}  \| \Delta \tilde{e}_{\textbf{m}}^{k-i} \|_{L^2}^2  + C  \|  \nabla  B_l( \tilde{e}_{\textbf{m}}^{k} ) \|_{L^2}^2 + C( C_0 \Delta t )^{2q} , \notag 
\end{align}
where $\epsilon$ is  an arbitrarily small positive constant which is independent of $\Delta t$ and will be determined below.
Using the H\"older  inequality in Lemma \ref{lem: Holder inequality} , the second term on the right\textcolor{black}{-}hand side of \eqref{e_error8} can be bounded by 
 \begin{equation}\label{e_error10}
\aligned
 | I_2 |  = &  \gamma |  \left(  | \nabla B_l( \textbf{m}( t^{k} ) ) |^2 B_l ( e_{\textbf{m}}^{k} ),   \Delta  \tilde{e}_{\textbf{m}}^{k+1} -  \tau_l \Delta  \tilde{e}_{\textbf{m}}^{k} \right)  |  \\
\leq & \gamma \| ( \nabla B_l( \textbf{m}( t^{k} ) ) )^2  \|_{L^3} \| B_l( e_{\textbf{m}}^{k} ) \|_{L^6} \|  \Delta  \tilde{e}_{\textbf{m}}^{k+1} -  \tau_l \Delta  \tilde{e}_{\textbf{m}}^{k} \|_{L^2} \\
\leq &  \gamma \| \nabla B_l( \textbf{m}( t^{k} ) ) \|_{L^6}^2 \| B_l( e_{\textbf{m}}^{k} ) \|_{L^6} \|  \Delta  \tilde{e}_{\textbf{m}}^{k+1} -  \tau_l \Delta  \tilde{e}_{\textbf{m}}^{k} \|_{L^2} \\
\leq & \epsilon  \|  \Delta  \tilde{e}_{\textbf{m}}^{k+1} \|_{L^2}^2 +  \epsilon \|  \Delta  \tilde{e}_{\textbf{m}}^{k} \|_{L^2}^2 + C  \| \nabla B_l(  \textbf{m}( t^{k} )  ) \|_{L^6}^4 \| \nabla B_l( \tilde{e}_{\textbf{m}}^{k} ) \|^2 _{L^2} + C( C_0 \Delta t )^{2q}.
\endaligned
\end{equation} 
For the third term on the right\textcolor{black}{-}hand side of \eqref{e_error8}, we have
 \begin{equation}\label{e_error11_I3}
\aligned
  | I_3 |  =&  \gamma \left(  |\nabla B_l( \textbf{m}( t^{k} ) ) |^2 B_l( \textbf{m}( t^{k} ) ) -  |\nabla \textbf{m}( t^{k+1} ) |^2 \textbf{m}( t^{k+1} ) ,  \Delta  \tilde{e}_{\textbf{m}}^{k+1}  - \tau_l \Delta  \tilde{e}_{\textbf{m}}^{k} ) \right) \\
= &  \gamma \left(  |\nabla B_l( \textbf{m}( t^{k} ) ) |^2 ( B_l( \textbf{m}( t^{k} ) ) - \textbf{m}( t^{k+1} ) ),   \Delta  \tilde{e}_{\textbf{m}}^{k+1}  - \tau_l \Delta  \tilde{e}_{\textbf{m}}^{k} )\right) \\
& + \gamma \left(  |\nabla B_l( \textbf{m}( t^{k} ) ) |^2 -  |\nabla \textbf{m}( t^{k+1} ) |^2  ,   \textbf{m}( t^{k+1} ) (\Delta  \tilde{e}_{\textbf{m}}^{k+1}  - \tau_l \Delta  \tilde{e}_{\textbf{m}}^{k} )  \right).
\endaligned
\end{equation}
Thus using the H\"older  inequality, the above term can be bounded by
 \begin{equation}\label{e_error11}
\aligned
  | I_3 |  \leq &  \epsilon \|  \Delta  \tilde{e}_{\textbf{m}}^{k+1} \|_{L^2}^2 +  \epsilon \|  \Delta  \tilde{e}_{\textbf{m}}^{k} \|_{L^2}^2  
 + C  \| \nabla B_l( \textbf{m}( t^{k} ) ) \|_{L^6}^2 \|  \sum\limits_{i=1}^l d_i \int_{ k+1-i}^{k+1} ( t^{ k+1-i} -s )^{l-1} \frac{ \partial ^{l} \textbf{m} }{\partial t ^l } (s) ds \|_{L^3} ^2  \\
& + C  ( \| \nabla B_l ( \textbf{m}( t^{k} ) )  \|_{L^6}^2 +  \| \nabla \textbf{m}( t^{k+1} ) \|_{L^6}^2  )
\|  \sum\limits_{i=1}^l d_i \int_{ k+1-i}^{k+1} ( t^{ k+1-i} -s )^{l-1} \frac{ \partial ^{l} \nabla \textbf{m} }{\partial t ^l } (s) ds \|_{L^3} ^2 \\
\leq  & \epsilon  \|  \Delta  \tilde{e}_{\textbf{m}}^{k+1} \|_{L^2}^2 + \epsilon  \|  \Delta  \tilde{e}_{\textbf{m}}^{k} \|_{L^2}^2  
+ + C ( \Delta t )^{2l-1}   \| \nabla \textbf{m}( t^{k+1} ) \|_{L^6}^2  \int_{ k+1-l }^{k+1} 
\|  \frac{ \partial ^{l} \textbf{m} }{\partial t ^l }(s) \|_{W^{1,3}}^2 ds  \\
& + C ( \Delta t )^{2l-1}  \| \nabla B_l ( \textbf{m}( t^{k} ) )  \|_{L^6}^2  \int_{ k+1-l }^{k+1} 
\|  \frac{ \partial ^{l} \textbf{m} }{\partial t ^l }(s) \|_{W^{1,3}}^2 ds ,
\endaligned
\end{equation}
where $ d_i $ with $i=1,2,\ldots,4,5$  are some fixed and bounded constants determined by the truncation errors. 

For  the fourth term on the right hand side of  \eqref{e_error8}, we have
  \begin{equation}\label{e_error12}
\aligned
  | I_4 |  =  \gamma | (\Delta \tilde{e}_{\textbf{m}}^{k+1},  \tau_l \Delta  \tilde{e}_{\textbf{m}}^{k} ) | \leq \frac{\gamma  \tau_l  }{2} |  \Delta  \tilde{e}_{\textbf{m}}^{k+1} \|_{L^2}^2 + \frac{ \gamma \tau_l }{2} |  \Delta  \tilde{e}_{\textbf{m}}^{k} \|_{L^2}^2.
\endaligned
\end{equation}
Recalling \eqref{e_error6}, the fifth term on the right hand side of  \eqref{e_error8} can be controlled by
  \begin{equation}\label{e_error12_R}
\aligned
  | I_5 |  = & | ( \textbf{R}_{\textbf{m},l}^{k+1}, \Delta  \tilde{e}_{\textbf{m}}^{k+1}  - \tau_l \Delta  \tilde{e}_{\textbf{m}}^{k} ) ) | \\
  \leq &   \epsilon  \|  \Delta  \tilde{e}_{\textbf{m}}^{k+1} \|_{L^2}^2 +  \epsilon  \|  \Delta  \tilde{e}_{\textbf{m}}^{k} \|_{L^2}^2  
   + C (\Delta t)^{2l-1}  \int_{ k+1-l }^{k+1} \|  \frac{ \partial ^{l+1} \textbf{m} }{\partial t ^{l+1} }(s) \|_{L^2}^2 ds. 
\endaligned
\end{equation}
\textcolor{black}{ 
The last term on the right hand side of  \eqref{e_error8} can be estimated by
 \begin{equation}\label{e_error_beta1}
\aligned
 | I_6 | = & | \beta | \left|  \left(   \textbf{m} (t^{k+1}) \times  \Delta\textbf{m} (t^{k+1}) - B_l( \textbf{m}^{k} )  \times \Delta B_l ( \tilde{\textbf{m}}^{k} ) , 
 \Delta  \tilde{e}_{\textbf{m}}^{k+1} - \tau_l \Delta  \tilde{e}_{\textbf{m}}^{k} \right) \right| \\
= & | \beta | \left|  \left(  \textbf{m} (t^{k+1}) \times ( \Delta \textbf{m} (t^{k+1}) - \Delta B_l( \textbf{m}(t^k) ) ) , \Delta  \tilde{e}_{\textbf{m}}^{k+1} - \tau_l \Delta  \tilde{e}_{\textbf{m}}^{k} \right) \right| \\
& +  | \beta | \left|  \left(  B_l( e_{\textbf{m}}^{k} ) \times \Delta B_l ( \tilde{\textbf{m}}^{k} )  , \Delta  \tilde{e}_{\textbf{m}}^{k+1} - \tau_l \Delta  \tilde{e}_{\textbf{m}}^{k} \right) \right| \\
& + | \beta | \left|  \left(   \textbf{m} (t^{k+1}) \times  \Delta B_l( \tilde{e}_{\textbf{m}}^{k} ) , \Delta  \tilde{e}_{\textbf{m}}^{k+1} - \tau_l \Delta  \tilde{e}_{\textbf{m}}^{k} \right) \right| \\
& + | \beta | \left|  \left(  ( \textbf{m} (t^{k+1}) - B_l( \textbf{m}(t^k) ) ) \times \Delta B_l ( \tilde{\textbf{m}}^{k} )  , \Delta  \tilde{e}_{\textbf{m}}^{k+1} - \tau_l \Delta  \tilde{e}_{\textbf{m}}^{k} \right) \right|.
\endaligned
\end{equation} 
Applying Cauchy-Schwarz inequality, the first term on the right hand side of \eqref{e_error_beta1} can be estimated by
 \begin{equation}\label{e_error_beta2}
\aligned
& | \beta | \left|  \left(  \textbf{m} (t^{k+1}) \times ( \Delta \textbf{m} (t^{k+1}) - \Delta B_l( \textbf{m}(t^k) ) ) , \Delta  \tilde{e}_{\textbf{m}}^{k+1} - \tau_l \Delta  \tilde{e}_{\textbf{m}}^{k} \right) \right|  \\
\leq & | \beta | \|  \sum\limits_{i=1}^l d_i \int_{ k+1-i}^{k+1} ( t^{ k+1-i} -s )^{l-1} \frac{ \partial ^{l} \Delta \textbf{m} }{\partial t ^l } (s) ds \|_{L^{2}} \| \textbf{m} (t^{k+1}) \|_{L^{\infty} } \|  \Delta  \tilde{e}_{\textbf{m}}^{k+1} - \tau_l \Delta  \tilde{e}_{\textbf{m}}^{k} \|_{L^2} \\
\leq &   \epsilon  \|  \Delta  \tilde{e}_{\textbf{m}}^{k+1} \|_{L^2}^2 +
\epsilon  \|  \Delta  \tilde{e}_{\textbf{m}}^{k} \|_{L^2}^2  +  C ( \Delta t )^{2l-1} \int_{ k+1-l }^{k+1} 
\|  \frac{ \partial ^{l} \textbf{m} }{\partial t ^l }(s) \|_{H^2}^2 ds .
\endaligned
\end{equation} 
Applying the interpolation inequality \eqref{e_Preliminaries3}, \eqref{e_High-order_correct} and \eqref{e_error4},  the second term on the right hand side of \eqref{e_error_beta1} can be bounded by
 \begin{equation}\label{e_error_beta3}
\aligned
& | \beta | \left|  \left(  B_l( e_{\textbf{m}}^{k} ) \times \Delta B_l ( \tilde{\textbf{m}}^{k} )  , \Delta  \tilde{e}_{\textbf{m}}^{k+1} - \tau_l \Delta  \tilde{e}_{\textbf{m}}^{k} \right) \right|  \\
\leq & \beta \| B_l( e_{\textbf{m}}^{k} ) \|_{L^{\infty}} \| \Delta B_l ( \tilde{\textbf{m}}^{k} ) \|_{L^2} \| \Delta  \tilde{e}_{\textbf{m}}^{k+1}  - \tau_l \Delta  \tilde{e}_{\textbf{m}}^{k}  \|_{L^2}  \\
\leq & C| \beta| \| B_l(  \tilde{e}_{\textbf{m}}^{k}  ) \|_{L^{\infty}} \| \Delta B_l ( \tilde{\textbf{m}}^{k} ) \|_{L^2} \| \Delta  \tilde{e}_{\textbf{m}}^{k+1}  - \tau_l \Delta  \tilde{e}_{\textbf{m}}^{k}  \|_{L^2}  \\
\leq & C |\beta | \| B_l(  \tilde{e}_{\textbf{m}}^{k}  ) \|_{H^1}^{1/2} \| B_l(  \tilde{e}_{\textbf{m}}^{k}  ) \|_{H^2}^{1/2}  \| \Delta B_l ( \tilde{\textbf{m}}^{k} ) \|_{L^2} \| \Delta  \tilde{e}_{\textbf{m}}^{k+1}  - \tau_l \Delta  \tilde{e}_{\textbf{m}}^{k}  \|_{L^2}  \\
\leq &  \epsilon  \|  \Delta  \tilde{e}_{\textbf{m}}^{k+1} \|_{L^2}^2 + \frac{ \epsilon }{ l }   \sum\limits_{i=0}^{l-1}  \| \Delta \tilde{e}_{\textbf{m}}^{k-i} \|_{L^2}^2 
 + C  \|  \nabla  B_l( \tilde{e}_{\textbf{m}}^{k} ) \|_{L^2}^2 .
\endaligned
\end{equation} 
Using Cauchy-Schwarz inequality,  the third term on the right hand side of \eqref{e_error_beta1} can be bounded by
 \begin{equation}\label{e_error_beta4}
\aligned
& | \beta | \left|  \left(   \textbf{m} (t^{k+1}) \times  \Delta B_l( \tilde{e}_{\textbf{m}}^{k} ) , \Delta  \tilde{e}_{\textbf{m}}^{k+1} - \tau_l \Delta  \tilde{e}_{\textbf{m}}^{k} \right) \right|   \\
\leq & | \beta | \|  \textbf{m} (t^{k+1}) \|_{L^{\infty}} \| \Delta B_l( \tilde{e}_{\textbf{m}}^{k} ) \|_{L^2} ( \| \Delta  \tilde{e}_{\textbf{m}}^{k+1}  \| + \tau_l \|  \Delta  \tilde{e}_{\textbf{m}}^{k}  \|_{L^2} ) \\
\leq & \frac{  |  \beta | \tau_l }{ 2 } \sum\limits_{i=0}^{l-1}  \| \Delta \tilde{e}_{\textbf{m}}^{k-i} \|_{L^2}^2 + \frac{ | \beta | }{ 2 } \| \Delta  \tilde{e}_{\textbf{m}}^{k+1} \| ^2 _{L^2}+ \frac{ | \beta | \tau_l }{ 2 } \| \Delta  \tilde{e}_{\textbf{m}}^{k} \| ^2 _{L^2}.
\endaligned
\end{equation} 
Applying the interpolation inequality \eqref{e_Preliminaries3}, \eqref{e_High-order_correct} and \eqref{e_error4},  the last term on the right hand side of \eqref{e_error_beta1} can be estimated by
 \begin{equation}\label{e_error_beta5}
\aligned
&  | \beta | \left|  \left(  ( \textbf{m} (t^{k+1}) - B_l( \textbf{m}(t^k) ) ) \times \Delta B_l ( \tilde{\textbf{m}}^{k} )  , \Delta  \tilde{e}_{\textbf{m}}^{k+1} - \tau_l \Delta  \tilde{e}_{\textbf{m}}^{k} \right) \right| \\
\leq & | \beta | \|  \sum\limits_{i=1}^l d_i \int_{ k+1-i}^{k+1} ( t^{ k+1-i} -s )^{l-1} \frac{ \partial ^{l} \textbf{m} }{\partial t ^l } (s) ds \|_{L^{\infty}} \| \Delta B_l ( \tilde{\textbf{m}}^{k} ) \|_{L^2} \| \Delta  \tilde{e}_{\textbf{m}}^{k+1} - \tau_l \Delta  \tilde{e}_{\textbf{m}}^{k} \|_{L^2} \\
\leq & \epsilon \|  \Delta  \tilde{e}_{\textbf{m}}^{k+1} \|_{L^2}^2 +  \epsilon  \|  \Delta  \tilde{e}_{\textbf{m}}^{k} \|_{L^2}^2 
 +  C ( \Delta t )^{2l-1} \int_{ k+1-l }^{k+1} \|  \frac{ \partial ^{l} \textbf{m} }{\partial t ^l }(s) \|_{H^2}^2 ds .
\endaligned
\end{equation} 
}

\textcolor{black}{
Finally combining \eqref{e_error8} with \eqref{e_error9}-\eqref{e_error_beta5}, we obtain
\begin{align}
& \sum\limits_{i,j=1}^l g_{i,j}
( \nabla \tilde{e}_{\textbf{m}}^{k+1+i-l } ,  \nabla \tilde{e}_{\textbf{m}}^{k+1+j-l } ) - \sum\limits_{i,j=1}^l g_{i,j}
( \nabla \tilde{e}_{\textbf{m}}^{k+i-l } ,  \nabla \tilde{e}_{\textbf{m}}^{k+j-l } ) \notag \\
& + \|  \sum\limits_{i=0}^l  \xi_i \nabla \tilde{e}_{\textbf{m}}^{k+1+i-l} \|^2 +  \frac{ \gamma (2 - \tau_l )- | \beta |   }{2} \Delta t \| \Delta \tilde{e}_{\textbf{m}}^{k+1} \|^2 \notag \\
\leq &  7 \epsilon  \Delta t \|  \Delta  \tilde{e}_{\textbf{m}}^{k+1} \|_{L^2}^2 
 + ( 5 \epsilon + \frac{ \gamma \tau_l + | \beta | \tau_l }{2} )
\Delta t \|  \Delta  \tilde{e}_{\textbf{m}}^{k} \|_{L^2}^2 \label{e_error14} \\
& +   \left(  \frac{  2 \epsilon  }{  l} +\frac{ | \beta | \tau_l  }{2} \right)  \sum\limits_{i=0}^{l-1}  \Delta t \| \Delta \tilde{e}_{\textbf{m}}^{k-i} \|_{L^2}^2 
+ C  \Delta t \|  \nabla  B_l( \tilde{e}_{\textbf{m}}^{k} ) \|_{L^2}^2 \notag \\
& + C \Delta t ( C_0 \Delta t )^{2q}  + C \Delta t \| \nabla B_l(  \textbf{m}( t^{k} )  ) \|_{L^6}^4 \| \nabla B_l( \tilde{e}_{\textbf{m}}^{k} ) \|^2 _{L^2} \notag \\
& + C ( \Delta t )^{2l} \int_{ k+1-l }^{k+1} 
\left( \|  \frac{ \partial ^{l} \textbf{m} }{\partial t ^l }(s) \|_{W^{1,3}}^2 + \|  \frac{ \partial ^{l+1} \textbf{m} }{\partial t ^{l+1} }(s) \|_{L^2}^2 +  \|  \frac{ \partial ^{l} \textbf{m} }{\partial t ^l }(s) \|_{H^2}^2 \right) ds , \notag
\end{align}
where $q=2$ with $l=1$ and $q=l$ with $l=2,3,4,5$ respectively.
}

Summing \eqref{e_error14} over $k$ from $0$ to $n$, and thanks to Lemma \ref{lem: multiplier step},  $G=(g_{i,j})$ is a symmetric positive definite matrix with minimum eigenvalue $\lambda_{\min}$, we obtain
\textcolor{black}{
  \begin{equation}\label{e_error15}
\aligned
& \frac{\lambda_{\min}}{ l }  \| \nabla \tilde{e}_{\textbf{m}}^{n+1 }\|^2 +  \left( \frac{  2 (1- \tau_l) - |\beta| (1+l\tau_l + \tau_l ) }{ 2 } - 14 \epsilon \right)  \Delta t  \sum\limits_{k=0}^{n} \| \Delta \tilde{e}_{\textbf{m}}^{k+1} \|^2 \\
\leq & C \sum\limits_{k=0}^{n}  \Delta t \|  \nabla  B_l( \tilde{e}_{\textbf{m}}^{k} ) \|_{L^2}^2 + C ( C_0 \Delta t )^{2q}  + C  \sum\limits_{k=0}^{n}  \Delta t \| \nabla B_l(  \textbf{m}( t^{k} )  ) \|_{L^6}^4 \| \nabla B_l( \tilde{e}_{\textbf{m}}^{k} ) \|^2 _{L^2} \\
& + C ( \Delta t )^{2l}  \int_{ 0}^{T} \|  \frac{ \partial ^{l} \textbf{m} }{\partial t ^l }(s) \|_{W^{1,3}}^2 ds  + C (\Delta t)^{2l}  \int_{ 0 }^{T} \|  \frac{ \partial ^{l+1} \textbf{m} }{\partial t ^{l+1} }(s) \|_{L^2}^2 ds.
  \endaligned
\end{equation}
Using the condition \eqref{e_beta_value1} and setting $ \epsilon = \frac{ 2 (1- \tau_l) - |\beta| (1+l\tau_l + \tau_l )  }{ 56 } $, we have
$$   \frac{  2 (1- \tau_l) - |\beta| (1+l\tau_l + \tau_l ) }{ 2 } - 14 \epsilon  =   \frac{  2 (1- \tau_l) - |\beta| (1+l\tau_l + \tau_l ) }{ 4 } >0 .$$
}
Then applying the discrete Gronwall Lemma \ref{lem: gronwall2} to \eqref{e_error15}, we can arrive at
 \begin{equation}\label{e_error_final_H2}
\aligned
& \|  \nabla \tilde{e}_{\textbf{m}}^{n+1} \|^2  + \Delta t \sum\limits_{k=0}^{n} \| \Delta \tilde{e}_{\textbf{m}}^{k+1} \|^2 
\leq 
\begin{cases}
C_1 \left(1+C_0^4 (\Delta t)^2 \right) (\Delta t)^2, \ \ l=1, \\
C_1 \left(1+C_0^{2l} \right) (\Delta t)^{2l}, \ \ l=2,3,4,5,
\end{cases}
\endaligned
\end{equation}  
where $C_1$ is independent of $C_0$ and $\Delta t$.

Thus, using the Poincar$\acute{e}$ inequality, we have 
\begin{equation}\label{e_error_final_induction4}
\aligned
& \|  \tilde{e}_{\textbf{m}}^{n+1} \|^2 + \|  \nabla \tilde{e}_{\textbf{m}}^{n+1} \|^2  + \Delta t \sum\limits_{k=0}^{n} \| \Delta \tilde{e}_{\textbf{m}}^{k+1} \|^2  \leq 
\begin{cases}
C_2 \left(1+C_0^4 (\Delta t)^2 \right) (\Delta t)^2, \ \ l=1, \\
C_2 \left(1+C_0^{2l} \right) (\Delta t)^{2l}, \ \ l=2,3,4,5,
\end{cases}
\endaligned
\end{equation} 
where $C_2$ is independent of $C_0$ and $\Delta t$.

\noindent{\bf Step 2: Estimates for $| 1-\xi^{n+1} |$.} 
 Let $\textbf{S}_{ R }^{k+1}$ be the truncation error defined by
\begin{equation}\label{e_error_xi_1}
\aligned
\textbf{S}_{ R }^{k+1}=\frac{\partial R (t^{k+1})}{\partial t}- \frac{ R (t^{k+1})-R(t^{k})}{\Delta t}=\frac{1}{\Delta t}\int_{t^k }^{t^{k+1}}(t^k-t)\frac{\partial^2 R }{\partial t^2}dt.
\endaligned
\end{equation}
Subtracting \eqref{e_High-order2} from \eqref{e_model_transform2} at $t^{k+1}$, we obtain the error equation corresponding to \eqref{e_High-order2}:
\begin{equation}\label{e_error_xi_2}
\aligned
& \frac{ e_{R}^{k+1} -  e_{R}^{k} }{ \Delta t } =  - \frac{R(t^{k+1} )}{ E( \textbf{m}(t^{k+1}) ) +K_0 } \| \textbf{m}(t^{k+1}) \times \Delta \textbf{m}(t^{k+1}) \|^2 \\
& \ \ \ \ \ \ \ \ \ \ \ \ 
+ \frac{R^{k+1} }{ E( \tilde{\textbf{m}}^{k+1} ) +K_0 }  \| B_l( \textbf{m}^{k} ) \times \Delta \tilde{\textbf{m}}^{k+1} \|^2 
- \textbf{S}_{ R }^{k+1} \\
=& - \xi^{n+1}  \left( \| \textbf{m}(t^{k+1}) \times \Delta \textbf{m}(t^{k+1}) \|^2 -  \| B_l( \textbf{m}^{k} ) \times \Delta \tilde{\textbf{m}}^{k+1} \|^2 \right) \\
& - R(t^{k+1}) \left(  \frac{1}{ E( \textbf{m}(t^{k+1}) ) +K_0 } - \frac{1}{ E( \tilde{\textbf{m}}^{k+1} ) +K_0 }  \right) \| \textbf{m}(t^{k+1}) \times \Delta \textbf{m}(t^{k+1}) \|^2 - \textbf{S}_{ R }^{k+1} \\
& -    \frac{e_R^{k+1} }{ E( \tilde{\textbf{m}}^{k+1} ) +K_0 } \| \textbf{m}(t^{k+1}) \times \Delta \textbf{m}(t^{k+1}) \|^2.
\endaligned
\end{equation}
Using \eqref{e_stability2} and the Sobolev embedding inequality (3.28) in \cite{liu2010stable},  the first term on the right\textcolor{black}{-}hand side of \eqref{e_error_xi_2} can be bounded by
\begin{equation}\label{e_error_xi_4}
\aligned
& - \xi^{n+1} \left( \| \textbf{m}(t^{k+1}) \times \Delta \textbf{m}(t^{k+1}) \|^2 -  \| B_l( \textbf{m}^{k} ) \times \Delta \tilde{\textbf{m}}^{k+1} \|^2 \right) \\
\leq & \xi^{n+1} \| \textbf{m}(t^{k+1}) \times \Delta \textbf{m}(t^{k+1}) + B_l( \textbf{m}^{k} ) \times \Delta \tilde{\textbf{m}}^{k+1} \| \| \textbf{m}(t^{k+1}) \times \Delta \textbf{m}(t^{k+1}) - B_l( \textbf{m}^{k} ) \times \Delta \tilde{\textbf{m}}^{k+1} \| \\
\leq & \xi^{n+1} \| \textbf{m}(t^{k+1}) \times \Delta \textbf{m}(t^{k+1}) + B_l( \textbf{m}^{k} ) \times \Delta \tilde{\textbf{m}}^{k+1} \| \| 
( \textbf{m}(t^{k+1}) - B_l( \textbf{m}^{k} ) )  \times \Delta \textbf{m}(t^{k+1})  \| \\
&+  \xi^{n+1} \| \textbf{m}(t^{k+1}) \times \Delta \textbf{m}(t^{k+1}) + B_l( \textbf{m}^{k} ) \times \Delta \tilde{\textbf{m}}^{k+1} \| \| B_l( \textbf{m}^{k} )   \times (\Delta \textbf{m}(t^{k+1}) - \Delta \tilde{\textbf{m}}^{k+1} ) \| \\
\leq & C \xi^{n+1} \| B_l( \textbf{m}^{k} ) \times \Delta \tilde{\textbf{m}}^{k+1}  \|_{L^2} \| B_l( e_{\textbf{m}}^{k} )  \|_{L^4} \| \Delta \textbf{m}(t^{k+1})  \|_{L^4} \\
& +  \frac{1}{4 M_T}  \xi^{n+1}  \| B_l( \textbf{m}^{k} ) \times \Delta \tilde{\textbf{m}}^{k+1}  \|_{L^2} \| B_l( \textbf{m}^{k} )  \|_{L^{\infty}} \| \Delta \tilde{e}_{\textbf{m}}^{k+1}  \|_{L^2} \\
& +  C\| \Delta \textbf{m}(t^{k+1})  \|_{L^{\infty}} \| \Delta  \tilde{e}_{\textbf{m}}^{k+1}  \|_{L^2}  + C \|  \sum\limits_{i=1}^l d_i \int_{ k+1-i}^{k+1} ( t^{ k+1-i} -s )^{l-1} \frac{ \partial ^{l} \textbf{m} }{\partial t ^l } (s) ds \|_{H^1}.
\endaligned
\end{equation}
The second term on the right\textcolor{black}{-}hand side of \eqref{e_error_xi_2} can be estimated by
\begin{equation}\label{e_error_xi_5}
\aligned
- R(t^{k+1})& \left(  \frac{1}{ E( \textbf{m}(t^{k+1}) ) +1 } - \frac{1}{ E( \tilde{\textbf{m}}^{k+1} ) +1 }  \right) \| \textbf{m}(t^{k+1}) \times \Delta \textbf{m}(t^{k+1}) \|^2 \\
\leq &C \| \textbf{m}(t^{k+1}) \times \Delta \textbf{m}(t^{k+1}) \|^2  ( \| \nabla \tilde{\textbf{m}}^{k+1} \|^2 - \| \nabla \textbf{m}(t^{k+1}) \|^2 ) \\
\leq & C \| \textbf{m}(t^{k+1}) \times \Delta \textbf{m}(t^{k+1}) \|^2 (  \nabla \tilde{\textbf{m}}^{k+1} + \nabla \textbf{m}(t^{k+1}))\| \nabla \tilde{e}_{\textbf{m}}^{k+1} \|
\endaligned
\end{equation}
Combining \eqref{e_error_xi_2} with \eqref{e_error_xi_4} and \eqref{e_error_xi_5}, and taking the inner product with $2 \Delta t e_{R}^{k+1}$ result in
\begin{equation}\label{e_error_xi_6}
\aligned
& ( | e_{R}^{k+1} |^2 - | e_{R}^{k} |^2 + | e_{R}^{k+1}- e_{R}^{k} |^2)  +2 \Delta t  \frac{ \| \textbf{m}(t^{k+1}) \times \Delta \textbf{m}(t^{k+1}) \|^2 }{ E( \tilde{\textbf{m}}^{k+1} ) +K_0 }   | e_R^{k+1} |^2 \\
\leq & C  \Delta t \xi^{n+1} \| B_l( \textbf{m}^{k} ) \times \Delta \tilde{\textbf{m}}^{k+1}  \|^2 ( \| B_l( \tilde{e}_{\textbf{m}}^{k} ) \|_{L^2} \| \nabla B_l( \tilde{e}_{\textbf{m}}^{k} ) \|_{L^2} + (C_0 \Delta t)^{2q} ) \\
&+  \frac{1}{4 M_T} \xi^{n+1}  \Delta t  \| B_l( \textbf{m}^{k} ) \times \Delta \tilde{\textbf{m}}^{k+1}  \|^2 \| B_l( \textbf{m}^{k} ) \|_{L^{\infty}}^2 |e_{R}^{k+1}| ^2  + C_3 \Delta t |e_R^{k+1}|^2 \\
& + C \Delta t  \xi^{n+1} \| \Delta  \tilde{e}_{\textbf{m}}^{k+1}  \|_{L^2}^2 + C ( \Delta t )^{2l} \int_{ k+1-l }^{k+1}  \|  \frac{ \partial ^{l} \textbf{m} }{\partial t ^l }(s) \|_{H^1}^2 ds + C (\Delta t)^3.
\endaligned
\end{equation}
Summing \eqref{e_error_xi_6}  over $k$, $k=0,1,2,\ldots,n^*$, where $n^*$ is the time step at which $|e_R^{n^*+1}|$ achieves its maximum value for $k=0,1,2,\ldots,n$, we can obtain
\begin{equation}\label{e_error_xi_7}
\aligned
  | e_{R}^{n^*+1} |^2 \leq &  \frac{1}{4 M_T}   |e_{R}^{n^*+1}| ^2  \sum_{k=0}^{n^*} \Delta t \xi^{n+1} \| B_l( \textbf{m}^{k} ) \times \Delta \tilde{\textbf{m}}^{k+1}  \|^2 \| \tilde{\textbf{m}}^{k}  \|_{L^{\infty}}^2 \\
&+ C \sum_{k=0}^{n^*} \Delta t \xi^{n+1} \| B_l( \textbf{m}^{k} ) \times \Delta \tilde{\textbf{m}}^{k+1}  \|^2 \| B_l( \tilde{e}_{\textbf{m}}^{k} )  \|_{L^2} ^2 \\
& + C \sum_{k=0}^{n^*} \Delta t  \xi^{n+1} \| B_l( \textbf{m}^{k} ) \times \Delta \tilde{\textbf{m}}^{k+1}  \|^2 \| \nabla B_l( \tilde{e}_{\textbf{m}}^{k} )  \|_{L^2} ^2+ C\sum_{k=0}^{n^*} \Delta t \| \Delta  \tilde{e}_{\textbf{m}}^{k+1}  \|_{L^2}^2  \\
& + C_3 \sum_{k=0}^{n^*} \Delta t |e_R^{k+1}|^2 + C ( \Delta t )^{2l} \int_{ 0}^{T}  \|  \frac{ \partial ^{l} \textbf{m} }{\partial t ^l }(s) \|_{H^1}^2 ds + C (C_0 \Delta t)^{2q} + C (\Delta t)^2.
\endaligned
\end{equation}
Thus  choosing $\Delta t \leq \frac{1}{2C_3}$ and  using the discrete Gronwall Lemma, we have
\begin{equation}\label{e_error_xi_final}
\aligned
 | e_{R}^{n+1} |^2 \leq & | e_{R}^{n^*+1} |^2 \le
 \begin{cases}
C_4 \left(1+C_0^4 (\Delta t)^2 \right) (\Delta t)^2, \ \ l=1, \\
C_4 \left(1+C_0^{2l} (\Delta t)^{2l-2} \right) (\Delta t)^{2}, \ \ l=2,3,4,5,
\end{cases}
\endaligned
\end{equation}
where $C_4$ is independent of $C_0$ and $\Delta t$.

\noindent{\bf Step 3: Completion of  the induction process.} Recalling \eqref{e_High-order2}, we have
\begin{equation}\label{e_error_final_induction1}
\aligned
 | 1- \xi^{n+1} | =  & | \frac{R(t^{k+1} )}{ E( \textbf{m}(t^{n+1}) ) +K_0 }  - \frac{R^{n+1} }{ E( \tilde{\textbf{m}}^{n+1} ) +K_0 } | \\
 \leq & C(  | e_{R}^{n+1} | + \| \nabla \tilde{e}_{\textbf{m}}^{n+1} \| ) \\
  \leq & C_5 \Delta t \sqrt{  1+C_0^{2q} (\Delta t)^{2q-2} },
\endaligned
\end{equation}
where $C_5$ is independent of $C_0$ and $\Delta t$,  $q=2$ when $l=1$, and $q=l$ when $l=2,3,4,5$.

Letting $C_0=\max\{ \frac{( 8C_1)^{4/q} }{2} ,    \frac{( 2C_2)^{2/q} }{2} , C_3 , 2 C_5,1 \}$ and $\Delta t \leq \frac{1}{ 1+2^{q+2} C_0^q }$,  we can obtain
\begin{equation}\label{e_error_final_induction2}
\aligned
 C_5 \sqrt{  1+C_0^{2q} (\Delta t)^{2q-2} } \leq C_5 (1+C_0^{q} \Delta t ) \leq C_0.
\endaligned
\end{equation}
Then combining \eqref{e_error_final_induction1} with \eqref{e_error_final_induction2} results in
\begin{equation}\label{e_error_final_induction3}
\aligned
 | 1- \xi^{n+1} |  \leq &C_0 \Delta t.
\endaligned
\end{equation}
Recalling \eqref{e_error_final_induction4}, we have
\begin{equation}\label{e_error_final_induction5}
\aligned
\|  \tilde{e}_{\textbf{m}}^{n+1} \| + \|  \nabla \tilde{e}_{\textbf{m}}^{n+1} \|  +  \| \Delta \tilde{e}_{\textbf{m}}^{n+1} \|  \leq &
\begin{cases}
 \sqrt{ C_2 \left(1+C_0^4 (\Delta t)^2 \right)  } (\Delta t)^{1/2}, \ \ l=1, \\
 \sqrt{ C_2 \left(1+C_0^{2l}  \right) (\Delta t)^{2l-2} }  (\Delta t)^{1/2}, \ \ l=2,3,4,5.
\end{cases}
\endaligned
\end{equation}

Thus we have
\begin{equation}\label{e_error_final_induction6}
\aligned
\|  \tilde{e}_{\textbf{m}}^{n+1} \| + \|  \nabla \tilde{e}_{\textbf{m}}^{n+1} \|  +  \| \Delta \tilde{e}_{\textbf{m}}^{n+1} \| 
& \leq \sqrt{2C_2 }  (\Delta t)^{1/2}  \leq (\Delta t)^{1/4} ,
\endaligned
\end{equation} 
which completes the induction process \eqref{e_error1} and \eqref{e_error2}.

Recalling \eqref{e_error_final_induction4} and Lemma \ref{lem: relation for error functions}, we have the final results
\begin{equation}\label{e_error_final_m}
\aligned
& \|  \textbf{m}^{n+1} - \textbf{m}(t^{n+1}) \|^2 + \|  \nabla( \textbf{m}^{n+1} - \textbf{m}(t^{n+1}) ) \|^2 + \|  \hat{\textbf{m}}^{n+1} - \textbf{m}(t^{n+1}) \|^2 \\
& + \|  \nabla( \hat{\textbf{m}}^{n+1} - \textbf{m}(t^{n+1}) ) \|^2 + \|  \tilde{\textbf{m}}^{n+1} - \textbf{m}(t^{n+1}) \|^2 \\
& + \|  \nabla( \tilde{\textbf{m}}^{n+1} - \textbf{m}(t^{n+1}) ) \|^2 +  \Delta t \sum\limits_{k=0}^{n} \|  \Delta( \tilde{\textbf{m}}^{k+1} - \textbf{m}(t^{k+1}) ) \|^2 
\leq  C(\Delta t)^{2l}.
\endaligned
\end{equation} 
The proof of Theorem \ref{the: error_estimate_final} is finally completed.
\end{proof}

 \section{Semi-implicit discretization for the term $\beta \textbf{m}\times \Delta \textbf{m}$}  
{\color{black}
In the last section, we established error estimates for the scheme \eqref{e_High-order1} where the term $\beta \textbf{m}\times \Delta \textbf{m}$ is treated explicitly under the condition \eqref{e_beta_value1}.  In order to relax this condition, we shall treat the term   $\beta \textbf{m}\times \Delta \textbf{m}$  semi-implicitly, namely, we replace  \eqref{e_High-order1} with $S=0$ by 
\begin{equation}\label{e_semi implicit1}
	\aligned
	D_l \tilde{\textbf{m}}^{n+1} = & \gamma \Delta \tilde{\textbf{m}}^{n+1}+ \gamma | \nabla B_l( \textbf{m}^{n})  |^2 B_l( \textbf{m}^{n} )- \beta B_l( \textbf{m}^{n} ) \times \Delta  \tilde{\textbf{m}}^{n+1}.
	\endaligned
\end{equation} 
We establish below an error analysis for the above scheme under a  condition which is less restrictive than \eqref{e_beta_value1}, at the price of having to solve a coupled elliptic system with variable coefficients at each time step.

 The main result od this section is stated below.
  }

\medskip
\begin{theorem}\label{the: error_estimate_final_semi}
 \textcolor{black}{
 Supposing that the damping parameter $\gamma$ satisfies
  \begin{equation} \label{e_beta_value2}
\gamma > \frac{ |\beta | \tau_l }{ 1- \tau_l }
  \end{equation}
 and assuming  $\textbf{m}\in H^{l+1}(0,T;\textbf{L}^2(\Omega))\bigcap H^{l}(0,T;\textbf{W}^{1,3}(\Omega)) \bigcap L^{\infty}(0,T; H^4(\Omega))$, then for the scheme \eqref{e_semi implicit1}, \eqref{e_High-order2}-\eqref{e_High-order_eta},  we have
 \begin{equation}  \label{e_final_semi}
\aligned
& \|  \textbf{m}^{n+1} - \textbf{m}(t^{n+1}) \|^2 + \|  \nabla( \textbf{m}^{n+1} - \textbf{m}(t^{n+1}) ) \|^2 + \|  \hat{\textbf{m}}^{n+1} - \textbf{m}(t^{n+1}) \|^2 \\
& + \|  \nabla( \hat{\textbf{m}}^{n+1} - \textbf{m}(t^{n+1}) ) \|^2 + \|  \tilde{\textbf{m}}^{n+1} - \textbf{m}(t^{n+1}) \|^2 \\
& + \|  \nabla( \tilde{\textbf{m}}^{n+1} - \textbf{m}(t^{n+1}) ) \|^2 +  \Delta t \sum\limits_{k=0}^{n} \|  \Delta( \tilde{\textbf{m}}^{k+1} - \textbf{m}(t^{k+1}) ) \|^2 \\
\leq & C(\Delta t)^{2l}, \ \ \ \forall \ 0\leq n\leq N-1,\quad 1\le l\le 5,
\endaligned
\end{equation} 
where $C$ is a positive constant  independent of $\Delta t$.
}
\end{theorem}

\begin{proof}
{\color{black}
The proof of Theorem \ref{the: error_estimate_final_semi} is exactly the same as Theorem \ref{the: error_estimate_final}, except for the highly nonlinear term with exchange parameter $\beta \ne 0$, so for the sake of brevity our focus will be on this. Similar to \eqref{e_error_beta1}, we can obtain 
 \begin{equation}\label{e_error_beta1_semi implicit}
\aligned
 | I_6 | 
= & | \beta | \left|  \left(  ( \textbf{m} (t^{k+1}) - B_l( \textbf{m}(t^k) ) ) \times \Delta \tilde{e}_{\textbf{m}}^{k+1} , \Delta  \tilde{e}_{\textbf{m}}^{k+1} - \tau_l \Delta  \tilde{e}_{\textbf{m}}^{k} \right) \right| \\
& +  | \beta | \left|  \left(  ( \textbf{m} (t^{k+1}) - B_l( \textbf{m}(t^k) ) ) \times \Delta \textbf{m} (t^{k+1}) , \Delta  \tilde{e}_{\textbf{m}}^{k+1} - \tau_l \Delta  \tilde{e}_{\textbf{m}}^{k} \right) \right| \\
& +  | \beta | \left|  \left(  B_l( e_{\textbf{m}}^{k} ) \times \Delta  \tilde{e}_{\textbf{m}}^{k+1} , \Delta  \tilde{e}_{\textbf{m}}^{k+1} - \tau_l \Delta  \tilde{e}_{\textbf{m}}^{k} \right) \right| 
 + | \beta | \left|  \left(  B_l( e_{\textbf{m}}^{k} ) \times \Delta  \textbf{m} (t^{k+1}) , \Delta  \tilde{e}_{\textbf{m}}^{k+1} - \tau_l \Delta  \tilde{e}_{\textbf{m}}^{k} \right) \right| \\
& + | \beta | \left|  \left(   \textbf{m} (t^{k+1}) \times  \Delta \tilde{e}_{\textbf{m}}^{k+1} , \Delta  \tilde{e}_{\textbf{m}}^{k+1} - \tau_l \Delta  \tilde{e}_{\textbf{m}}^{k} \right) \right| .
\endaligned
\end{equation} 
Applying Cauchy-Schwarz inequality, the first term on the right hand side of \eqref{e_error_beta1_semi implicit} can be estimated by
 \begin{equation}\label{e_error_beta2_semi implicit}
\aligned
& \beta \left|  \left(  ( \textbf{m} (t^{k+1}) - B_l( \textbf{m}(t^k) ) ) \times \Delta \tilde{e}_{\textbf{m}}^{k+1} , \Delta  \tilde{e}_{\textbf{m}}^{k+1} - \tau_l \Delta  \tilde{e}_{\textbf{m}}^{k} \right) \right| \\
\leq & | \beta | \|  \sum\limits_{i=1}^l d_i \int_{ k+1-i}^{k+1} ( t^{ k+1-i} -s )^{l-1} \frac{ \partial ^{l} \textbf{m} }{\partial t ^l } (s) ds \|_{L^{\infty}} \| \Delta \tilde{e}_{\textbf{m}}^{k+1} \|_{L^2} \|  \tau_l \Delta  \tilde{e}_{\textbf{m}}^{k} \|_{L^2} \\
\leq &  \epsilon \|  \Delta  \tilde{e}_{\textbf{m}}^{k} \|_{L^2}^2
+  C ( \Delta t )^{2l-1} \int_{ k+1-l }^{k+1} 
\|  \frac{ \partial ^{l} \textbf{m} }{\partial t ^l }(s) \|_{H^2}^2 ds \|  \Delta  \tilde{e}_{\textbf{m}}^{k+1} \|_{L^2}^2 \\
\leq &  \epsilon  \|  \Delta  \tilde{e}_{\textbf{m}}^{k} \|_{L^2}^2 +
 \epsilon  \|  \Delta  \tilde{e}_{\textbf{m}}^{k+1} \|_{L^2}^2,
\endaligned
\end{equation} 
where we set $\Delta t \leq \hat{C} $, which is sufficiently small to keep the last inequality holds.

Applying the interpolation inequality \eqref{e_Preliminaries3},  the second term on the right hand side of \eqref{e_error_beta1_semi implicit} can be bounded by
 \begin{equation}\label{e_error_beta3_semi implicit}
\aligned
& \beta \left|  \left(  ( \textbf{m} (t^{k+1}) - B_l( \textbf{m}(t^k) ) ) \times \Delta \textbf{m} (t^{k+1}) , \Delta  \tilde{e}_{\textbf{m}}^{k+1} - \tau_l \Delta  \tilde{e}_{\textbf{m}}^{k} \right) \right| \\
\leq & | \beta | \|  \sum\limits_{i=1}^l d_i \int_{ k+1-i}^{k+1} ( t^{ k+1-i} -s )^{l-1} \frac{ \partial ^{l} \textbf{m} }{\partial t ^l } (s) ds \|_{L^{\infty}} \| \Delta \textbf{m} (t^{k+1}) \|_{L^2} \| \Delta  \tilde{e}_{\textbf{m}}^{k+1} - \tau_l \Delta  \tilde{e}_{\textbf{m}}^{k} \|_{L^2} \\
\leq &   \epsilon \|  \Delta  \tilde{e}_{\textbf{m}}^{k+1} \|_{L^2}^2 +   \epsilon \|  \Delta  \tilde{e}_{\textbf{m}}^{k} \|_{L^2}^2 
 +  C ( \Delta t )^{2l-1} \int_{ k+1-l }^{k+1} \|  \frac{ \partial ^{l} \textbf{m} }{\partial t ^l }(s) \|_{H^2}^2 ds .
\endaligned
\end{equation} 
Using the interpolation inequality \eqref{e_Preliminaries3} and \eqref{e_error4},  the third and fourth terms on the right hand side of \eqref{e_error_beta1_semi implicit} can be transformed into
 \begin{equation}\label{e_error_beta4_semi implicit}
\aligned
& \beta \left|  \left(  B_l( e_{\textbf{m}}^{k} ) \times \Delta  \tilde{e}_{\textbf{m}}^{k+1} , \Delta  \tilde{e}_{\textbf{m}}^{k+1} - \tau_l \Delta  \tilde{e}_{\textbf{m}}^{k} \right) \right| +  \beta \left|  \left(  B_l( e_{\textbf{m}}^{k} ) \times \Delta  \textbf{m} (t^{k+1}) , \Delta  \tilde{e}_{\textbf{m}}^{k+1} - \tau_l \Delta  \tilde{e}_{\textbf{m}}^{k} \right) \right| \\
\leq  &  \beta \| B_l( e_{\textbf{m}}^{k} ) \|_{L^{\infty}} \| \Delta  \tilde{e}_{\textbf{m}}^{k+1} \|_{L^2} \| \tau_l \Delta  \tilde{e}_{\textbf{m}}^{k}  \|_{L^2} 
+ \beta \| B_l( e_{\textbf{m}}^{k} ) \|_{L^{\infty}} \| \Delta \textbf{m} (t^{k+1}) \|_{L^2} \| \Delta  \tilde{e}_{\textbf{m}}^{k+1} - \tau_l \Delta  \tilde{e}_{\textbf{m}}^{k}  \|_{L^2} \\
\leq & \beta \| B_l( e_{\textbf{m}}^{k} ) \|_{H^1}^{1/2} \| B_l( e_{\textbf{m}}^{k} ) \|_{H^2}^{1/2}  ( \| \Delta  \tilde{e}_{\textbf{m}}^{k+1} \|_{L^2} \| \tau_l \Delta  \tilde{e}_{\textbf{m}}^{k}  \|_{L^2} +  \| \Delta  \textbf{m} (t^{k+1}) \|_{L^2} \|  \Delta  \tilde{e}_{\textbf{m}}^{k+1} - \tau_l \Delta  \tilde{e}_{\textbf{m}}^{k}  \|_{L^2} ) \\
\leq &   \epsilon\|  \Delta  \tilde{e}_{\textbf{m}}^{k+1} \|_{L^2}^2 + \frac{  \epsilon  }{  l}  \sum\limits_{i=0}^{l-1}  \| \Delta \tilde{e}_{\textbf{m}}^{k-i} \|_{L^2}^2 
 +  C (\Delta t)^{1/4} \|  \Delta  \tilde{e}_{\textbf{m}}^{k} \|_{L^2}^2 + C  \|  \nabla  B_l( \tilde{e}_{\textbf{m}}^{k} ) \|_{L^2}^2 \\
\leq &  \epsilon \|  \Delta  \tilde{e}_{\textbf{m}}^{k+1} \|_{L^2}^2 + \frac{  \epsilon }{ l }  \sum\limits_{i=0}^{l-1}  \| \Delta \tilde{e}_{\textbf{m}}^{k-i} \|_{L^2}^2 
 +    \epsilon \|  \Delta  \tilde{e}_{\textbf{m}}^{k} \|_{L^2}^2 + C  \|  \nabla  B_l( \tilde{e}_{\textbf{m}}^{k} ) \|_{L^2}^2,
\endaligned
\end{equation} 
where we set $\Delta t \leq \check{C}$, which is sufficiently small to keep the last inequality holds.

Recalling the fact that $ \| \textbf{m} (t^{k+1})  \|_{L^{\infty }} =1 $, the last term on the right hand side of \eqref{e_error_beta1_semi implicit} can be estimated by
 \begin{equation}\label{e_error_beta5_semi implicit}
\aligned
&\beta \left|  \left(   \textbf{m} (t^{k+1}) \times  \Delta \tilde{e}_{\textbf{m}}^{k+1} , \Delta  \tilde{e}_{\textbf{m}}^{k+1} - \tau_l \Delta  \tilde{e}_{\textbf{m}}^{k} \right) \right| =  \beta \left|  \left(   \textbf{m} (t^{k+1}) \times  \Delta \tilde{e}_{\textbf{m}}^{k+1} , - \tau_l \Delta  \tilde{e}_{\textbf{m}}^{k} \right) \right| \\
\leq  & | \beta | \|  \textbf{m} (t^{k+1}) \| _{L^{\infty}} \| \tau_l ^{1/2} \Delta \tilde{e}_{\textbf{m}}^{k+1} \|_{L^2} \| \tau_l ^{1/2} \Delta \tilde{e}_{\textbf{m}}^{k} \|_{L^2} \\
\leq & \frac{ | \beta | \tau_l }{ 2 } \|  \Delta \tilde{e}_{\textbf{m}}^{k+1} \|_{L^2}^2 + \frac{ | \beta | \tau_l }{ 2 } \|  \Delta \tilde{e}_{\textbf{m}}^{k} \|_{L^2}^2 .
\endaligned
\end{equation}
Thus by using the similar procedure for the fully explicit case in \eqref{e_error14}, we can obtain that 
  \begin{equation}\label{e_error14_semi implicit}
\aligned
& \sum\limits_{i,j=1}^l g_{i,j}
( \nabla \tilde{e}_{\textbf{m}}^{k+1+i-l } ,  \nabla \tilde{e}_{\textbf{m}}^{k+1+j-l } ) - \sum\limits_{i,j=1}^l g_{i,j}
( \nabla \tilde{e}_{\textbf{m}}^{k+i-l } ,  \nabla \tilde{e}_{\textbf{m}}^{k+j-l } ) \\
& + \|  \sum\limits_{i=0}^l  \xi_i \nabla \tilde{e}_{\textbf{m}}^{k+1+i-l} \|^2 +  \frac{ \gamma( 2 - \tau_l) - |\beta| \tau_l }{2} \Delta t \| \Delta \tilde{e}_{\textbf{m}}^{k+1} \|^2 \\
\leq & 7 \epsilon  \Delta t \|  \Delta  \tilde{e}_{\textbf{m}}^{k+1} \|_{L^2}^2 + (   6 \epsilon + \frac{ \gamma \tau_l + | \beta | \tau_l }{2} )
\Delta t \|  \Delta  \tilde{e}_{\textbf{m}}^{k} \|_{L^2}^2  +  \frac{ 2 \epsilon  }{  l}  \sum\limits_{i=0}^{l-1}  \Delta t \| \Delta \tilde{e}_{\textbf{m}}^{k-i} \|_{L^2}^2 \\
& + C  \Delta t \|  \nabla  B_l( \tilde{e}_{\textbf{m}}^{k} ) \|_{L^2}^2  + C \Delta t ( C_0 \Delta t )^{2q}  + C \Delta t \| \nabla B_l(  \textbf{m}( t^{k} )  ) \|_{L^6}^4 \| \nabla B_l( \tilde{e}_{\textbf{m}}^{k} ) \|^2 _{L^2} \\
& + C ( \Delta t )^{2l} \int_{ k+1-l }^{k+1} 
\left( \|  \frac{ \partial ^{l} \textbf{m} }{\partial t ^l }(s) \|_{W^{1,3}}^2 + \|  \frac{ \partial ^{l+1} \textbf{m} }{\partial t ^{l+1} }(s) \|_{L^2}^2 +  \|  \frac{ \partial ^{l} \textbf{m} }{\partial t ^l }(s) \|_{H^2}^2 \right) ds ,
  \endaligned
\end{equation}
Using the condition \eqref{e_beta_value2} and setting $ \epsilon = \frac{ \gamma ( 1- \tau_l ) - | \beta |  \tau_l  }{ 30 }  $, we have
$$  \gamma ( 1- \tau_l ) - | \beta | \tau_l -15 \epsilon >0 . $$
Summing \eqref{e_error14_semi implicit} over $k$ from $0$ to $n$ and applying the discrete Gronwall Lemma \ref{lem: gronwall2}, we can arrive at
 \begin{equation}\label{e_error_final_H2_semi implicit}
\aligned
& \|  \nabla \tilde{e}_{\textbf{m}}^{n+1} \|^2  + \Delta t \sum\limits_{k=0}^{n} \| \Delta \tilde{e}_{\textbf{m}}^{k+1} \|^2 
\leq 
\begin{cases}
C_1 \left(1+C_0^4 (\Delta t)^2 \right) (\Delta t)^2, \ \ l=1, \\
C_1 \left(1+C_0^{2l} \right) (\Delta t)^{2l}, \ \ l=2,3,4,5,
\end{cases}
\endaligned
\end{equation}  
where $C_1$ is independent of $C_0$ and $\Delta t$. Then using exactly the similar procedure in Steps 2 and 3 in Section 4, we can easily obtain the desired result \eqref{e_final_semi}. 
}
\end{proof}

 \section{Numerical experiments}
In this section, we carry out some numerical experiments to verify the accuracy and stability of the  high-order IMEX-GSAV schemes  \eqref{e_High-order1}-\eqref{e_High-order_eta} for the Landau-Lifshitz equation. We assume periodic boundary conditions and use the Fourier-spectral method for spatial discretization, which reduces the equation 
\eqref{e_High-order1} to a diagonal system in the frequency space so it can be easily implemented.

\subsection{Convergence rate with a known exact solution}
We test the convergence rate for the Landau-Lifshitz equation \eqref{e_LLmodel} with an external force so that the exact solution is
\begin{equation}
	\begin{aligned}
		m_1^e(x,y,t)=&\sin(t+x)\cos(t+y),\\
		m_2^e(x,y,t)=&\cos(t+x)\cos(t+y),\\
		m_3^e(x,y,t)=&\sin(t+y).\\
	\end{aligned}
\end{equation}

We set $S=0$, $K_0=1$, $w=1$ and  $\Omega= [0, 2\pi)^2$ with periodic boundary conditions and use the Fourier spectral method with
$64\times64$ modes for spatial approximation so that the spatial discretization error is negligible. 
For the case of $\beta=0$, we plot in Figures \ref{error1a}-\ref{error1d} the convergence rates in $l^{\infty}(0,T;H^1(\Omega)) \bigcap l^{2}(0,T;H^2(\Omega))$ with $l=1,2,3,4$ at $T=0.5$, which are in good agreement with  Theorem \ref{the: error_estimate_final}.  
We also plot  convergence rates for the first- to fourth-order schemes in Figures \ref{error3a}-\ref{error3d} for the case of $\beta=0.5$.

\begin{figure}[htp]
	\centering
	\subfigure[BDF1 vs. errors]{
		\begin{minipage}[c]{0.4\textwidth}
			\includegraphics[width=1\textwidth]{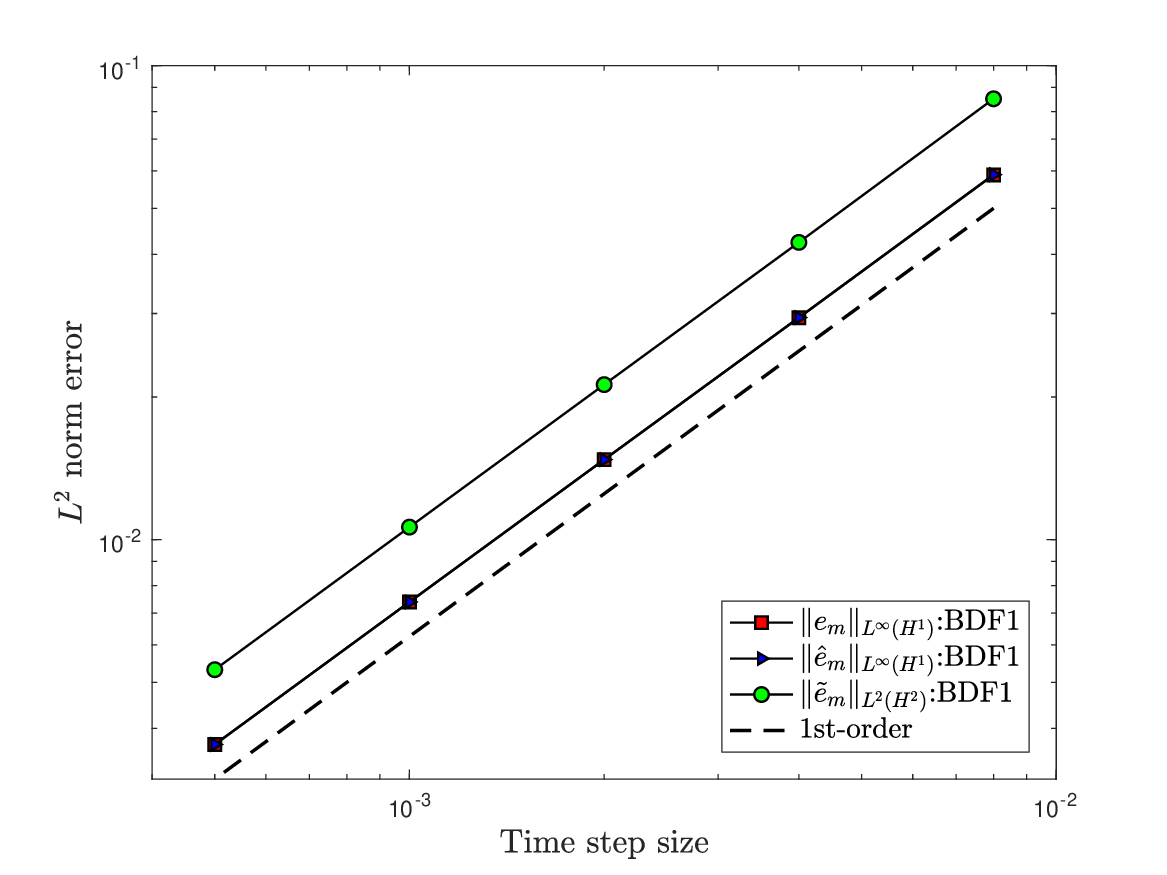}\label{error1a}
		\end{minipage}
	}
	\subfigure[BDF2 vs. errors]{
		\begin{minipage}[c]{0.4\textwidth}
			\includegraphics[width=1\textwidth]{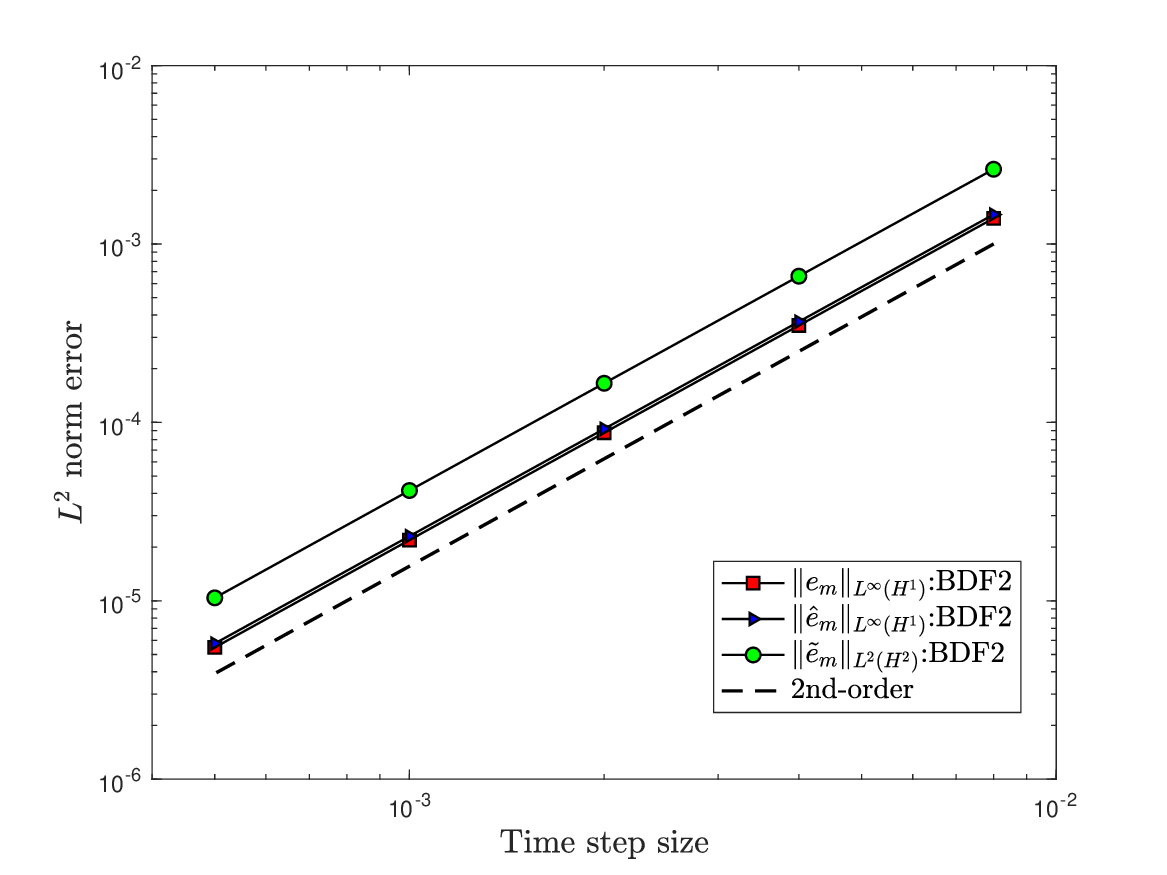}\label{error1b}
		\end{minipage}
	}
	\subfigure[BDF3 vs. errors]{
		\begin{minipage}[c]{0.4\textwidth}
			\includegraphics[width=1\textwidth]{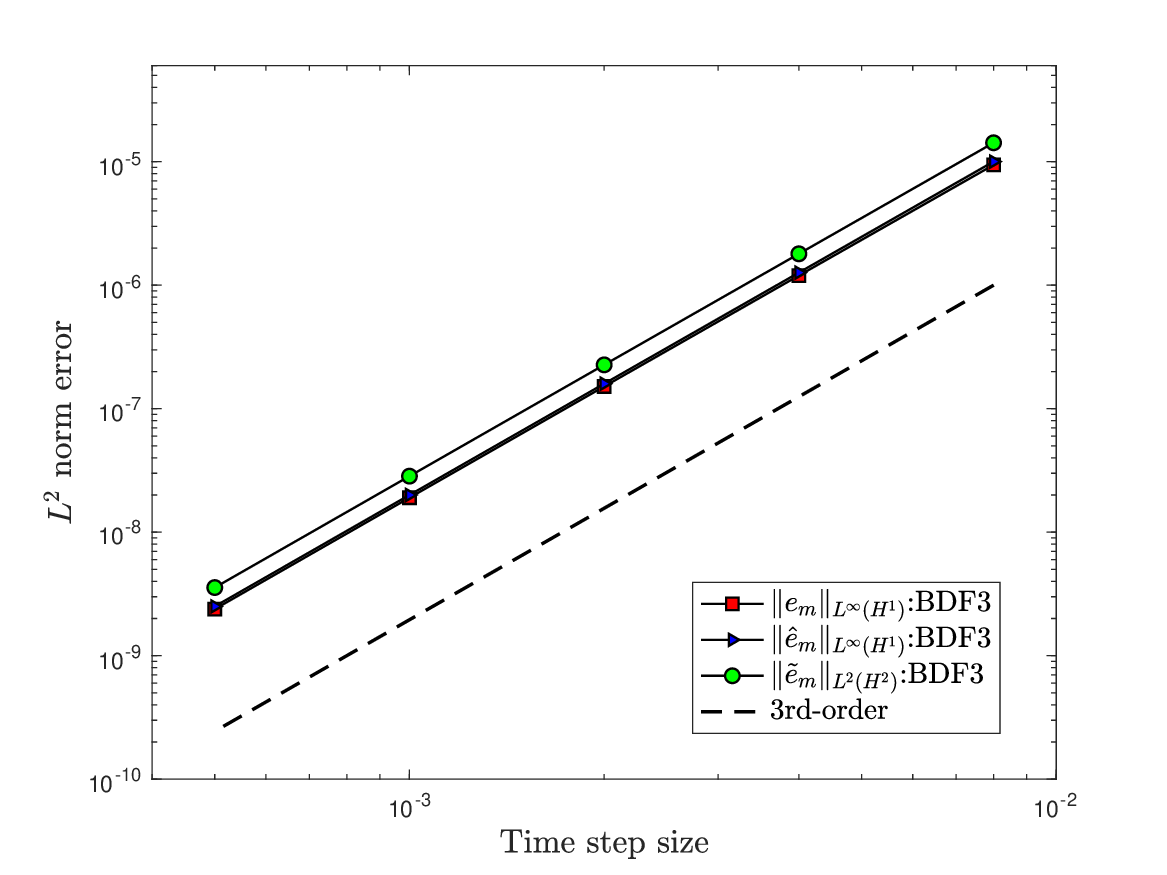}\label{error1c}
		\end{minipage}
	}
	\subfigure[BDF4 vs. errors]{
		\begin{minipage}[c]{0.4\textwidth}
			\includegraphics[width=1\textwidth]{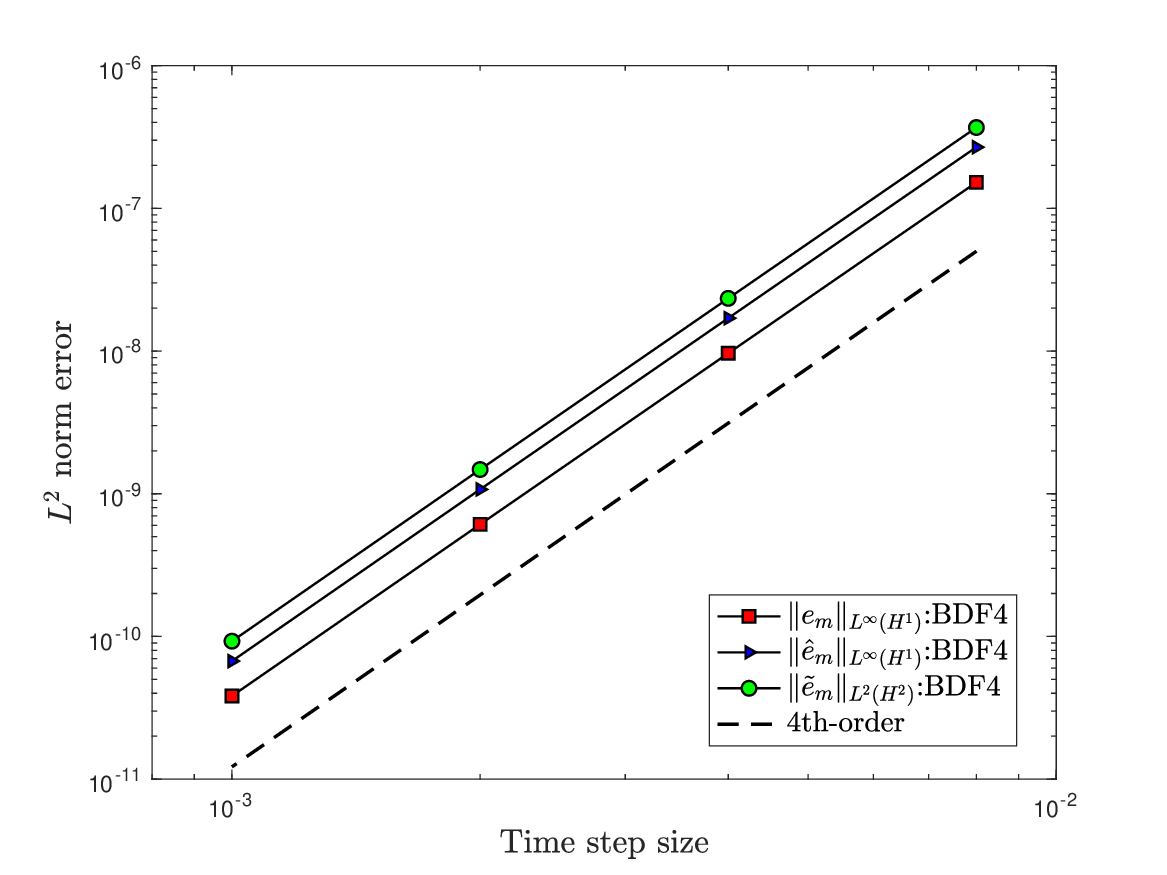}\label{error1d}
		\end{minipage}
	}
	{\caption{Numerical convergence rate of the first- to fourth-order schemes with $\beta=0$ in Example 1.}\label{error1}}
\end{figure}

\begin{figure}[htp]
	\centering
	\subfigure[BDF1 vs. errors]{
		\begin{minipage}[c]{0.4\textwidth}
			\includegraphics[width=1\textwidth]{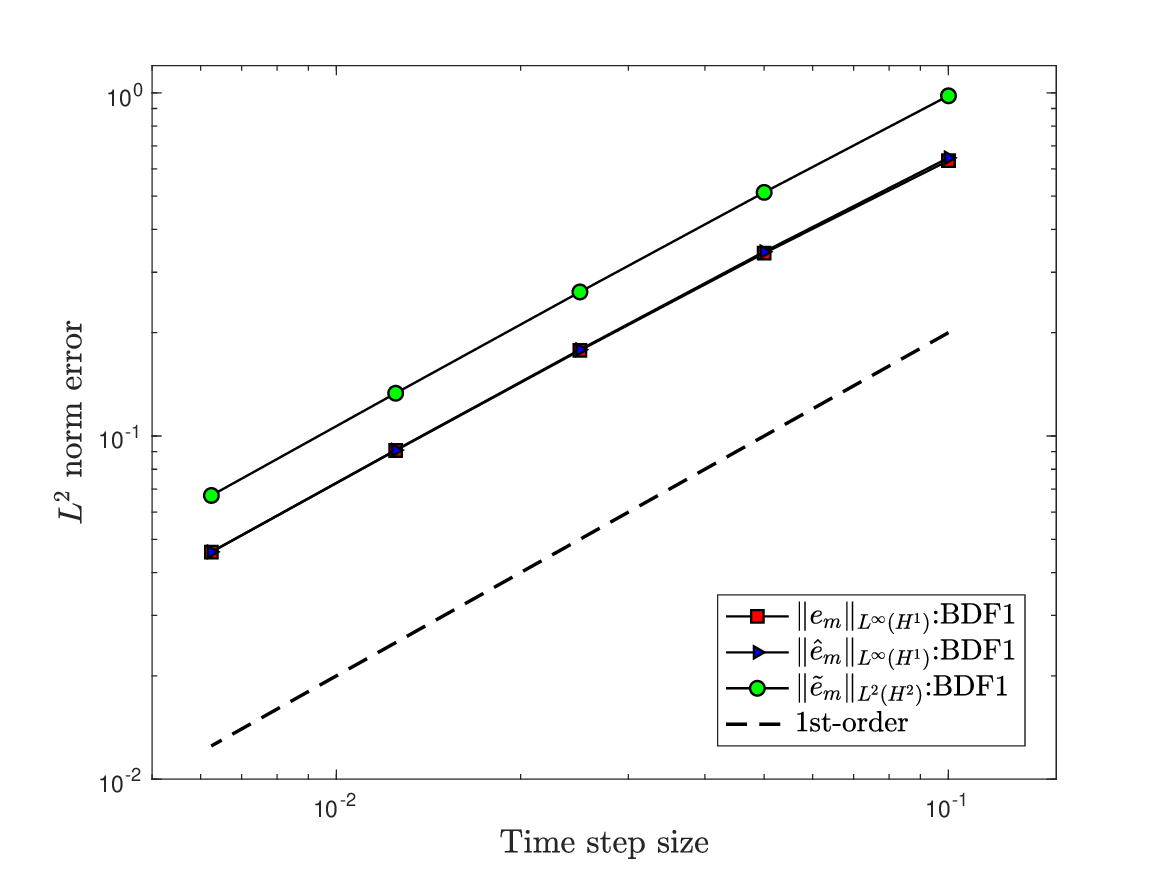}\label{error3a}
		\end{minipage}
	}
	\subfigure[BDF2 vs. errors]{
		\begin{minipage}[c]{0.4\textwidth}
			\includegraphics[width=1\textwidth]{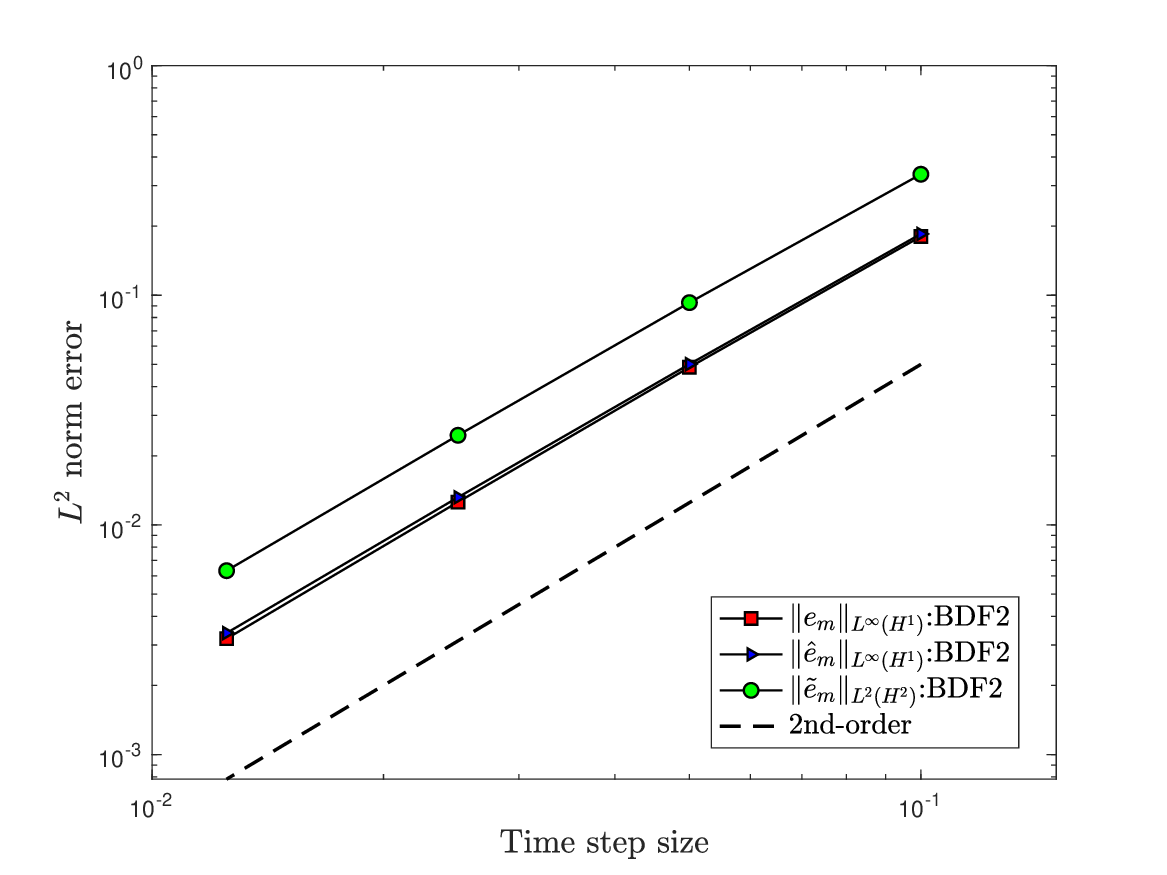}\label{error3b}
		\end{minipage}
	}
	\subfigure[BDF3 vs. errors]{
		\begin{minipage}[c]{0.4\textwidth}
			\includegraphics[width=1\textwidth]{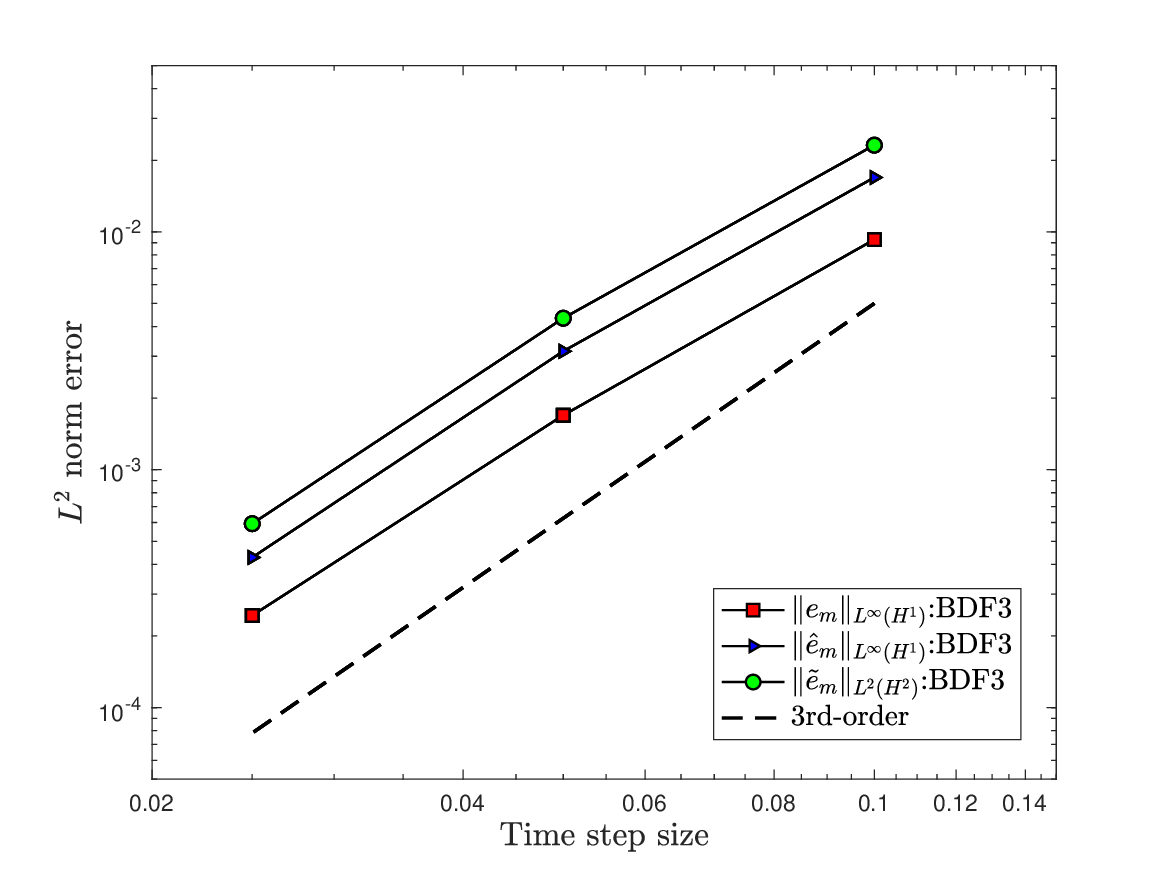}\label{error3c}
		\end{minipage}
	}
	\subfigure[BDF4 vs. errors]{
		\begin{minipage}[c]{0.4\textwidth}
			\includegraphics[width=1\textwidth]{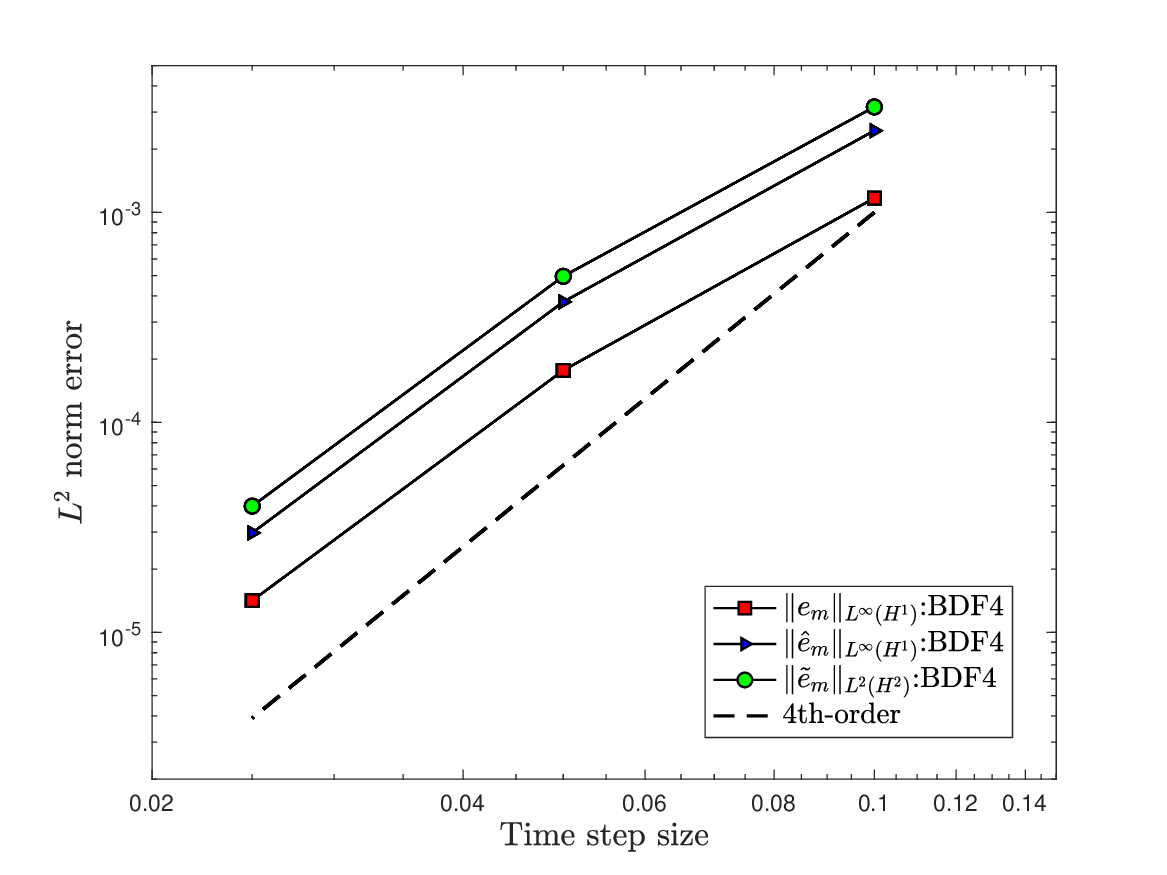}\label{error3d}
		\end{minipage}
	}
	{\caption{Numerical convergence rate of the first- to fourth-order schemes with $\beta=0.5$ in Example 1.}\label{error3}}
\end{figure}

\subsection{Convergence rate with an unknown exact solution}

We consider the Landau-Lifshitz equation with the initial condition
\begin{equation}
	\begin{aligned}
		m_1(x,y,0)=&\cos(x)\cos(y)\sin(0.1),\\
		m_2(x,y,0)=&\cos(x)\cos(y)\cos(0.1),\\
		m_3(x,y,0)=&(1-\cos^2(x)\cos^2(y))^{1/2},\\
	\end{aligned}
\end{equation}
in $\Omega= [0, 2\pi)^2$.
We set $S=0$, $\beta=0$, $K_0=1$, $w=2$  and  also use $256$ Fourier modes in each direction for spatial approximation.
The exact solution is unknown so we compute a reference solution by the fourth-order method with a small time step $\Delta t=3.125E-6$. 
 We plot the convergence rates  in $l^{\infty}(0,T;H^1(\Omega)) \bigcap l^{2}(0,T;H^2(\Omega))$ norm for the projected and unprojected magnetization by using first- to fourth-order schemes in Figures \ref{error2a}-\ref{error2d} respectively. We observe that the expected orders of convergence are achieved, although the errors in $ l^{2}(0,T;H^2(\Omega))$ are several orders of magnitude larger than those in $l^{\infty}(0,T;H^1(\Omega)) $.

\begin{figure}[htp]
	\centering
	\subfigure[BDF1 vs. errors]{
		\begin{minipage}[c]{0.4\textwidth}
			\includegraphics[width=1\textwidth]{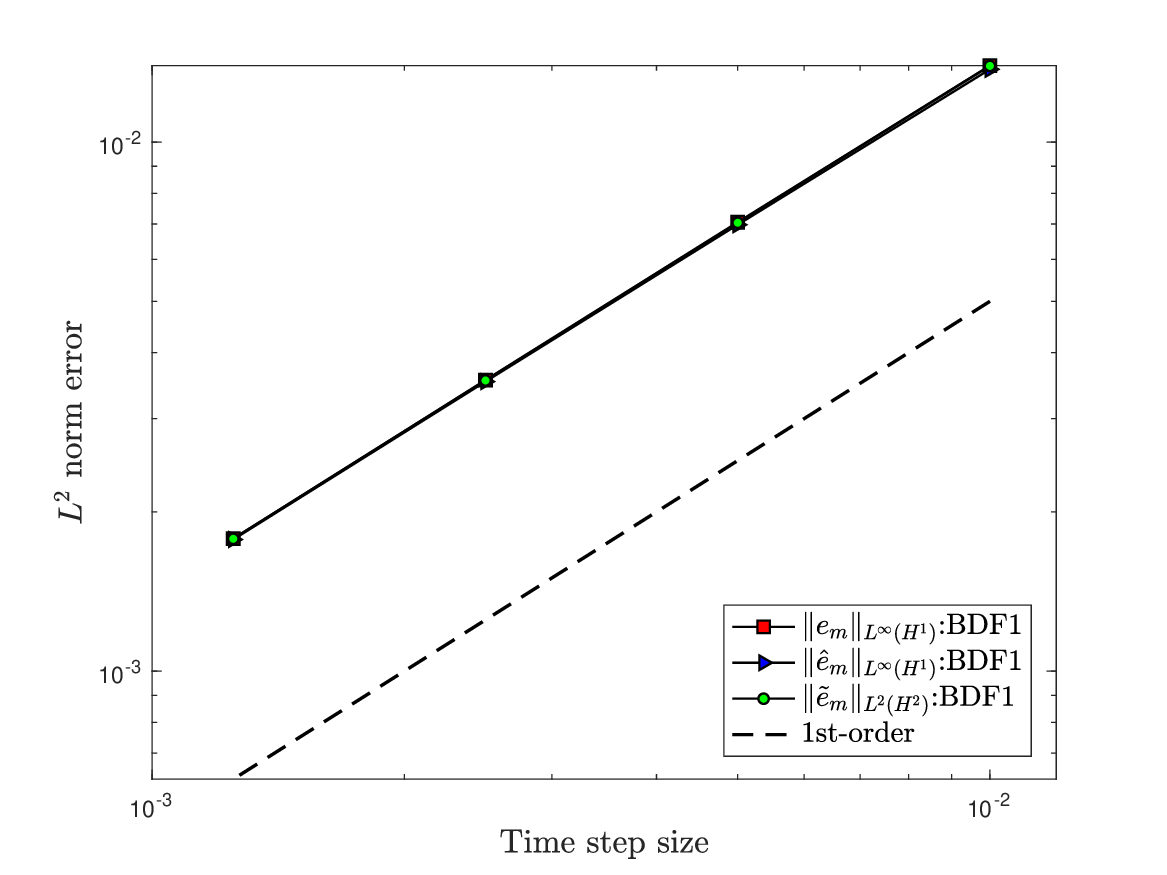}\label{error2a}
		\end{minipage}
	}
	\subfigure[BDF2 vs. errors]{
		\begin{minipage}[c]{0.4\textwidth}
			\includegraphics[width=1\textwidth]{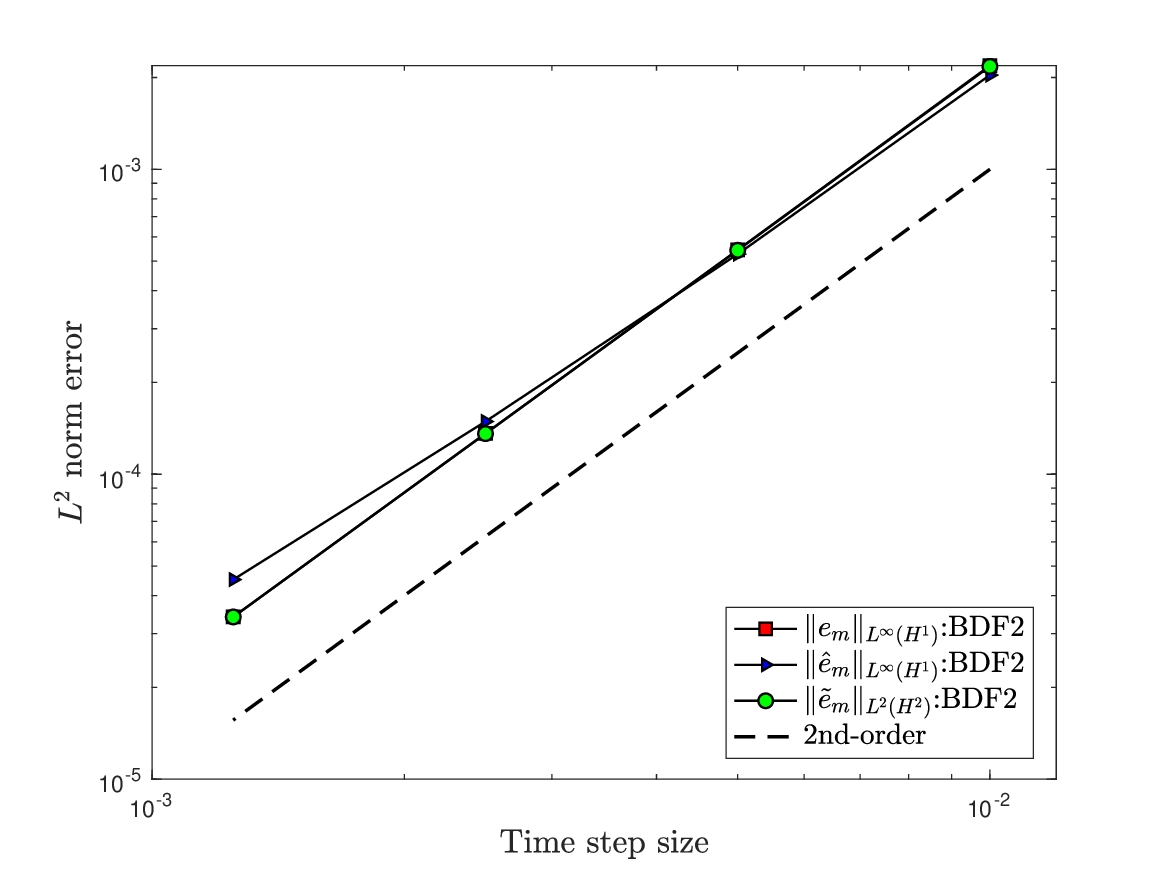}\label{error2b}
		\end{minipage}
	}
		\subfigure[BDF3 vs. errors]{
				\begin{minipage}[c]{0.4\textwidth}
						\includegraphics[width=1\textwidth]{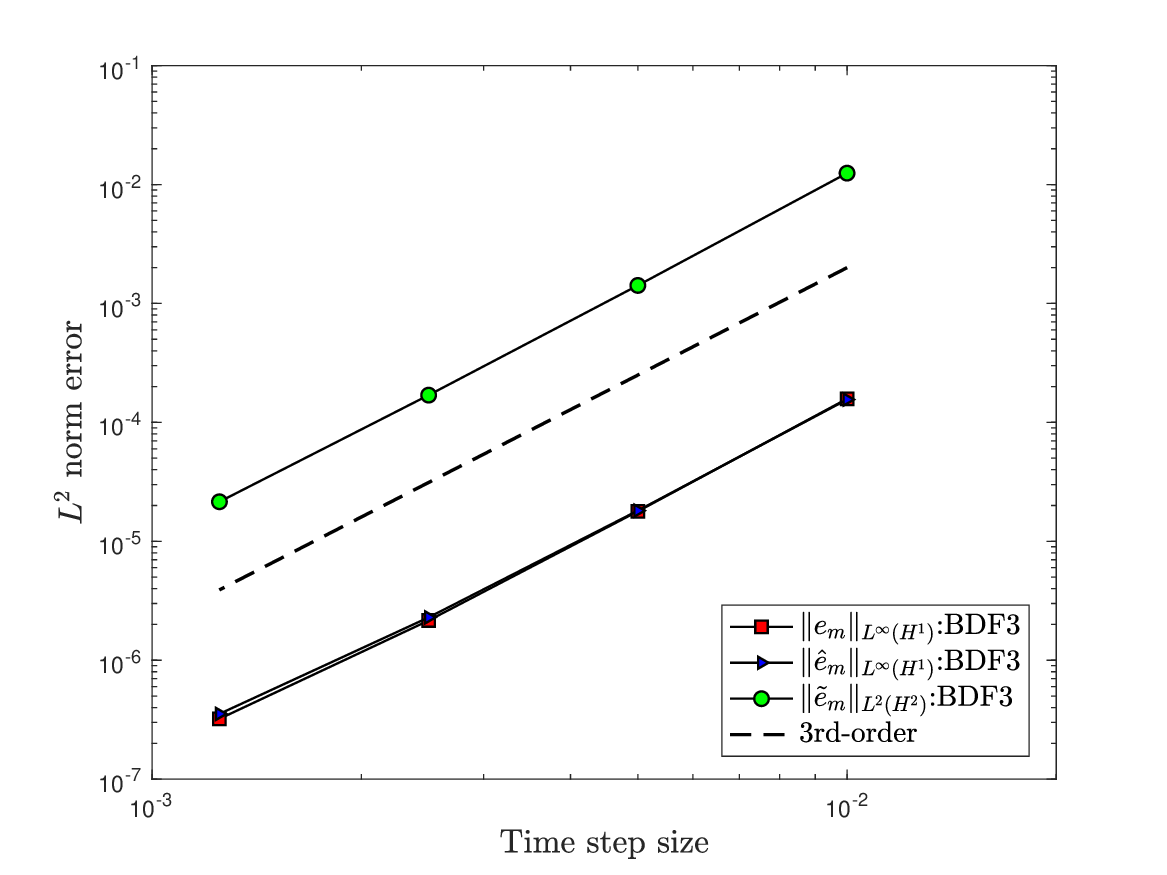}\label{error2c}
					\end{minipage}
			}
		\subfigure[BDF4 vs. errors]{
				\begin{minipage}[c]{0.4\textwidth}
						\includegraphics[width=1\textwidth]{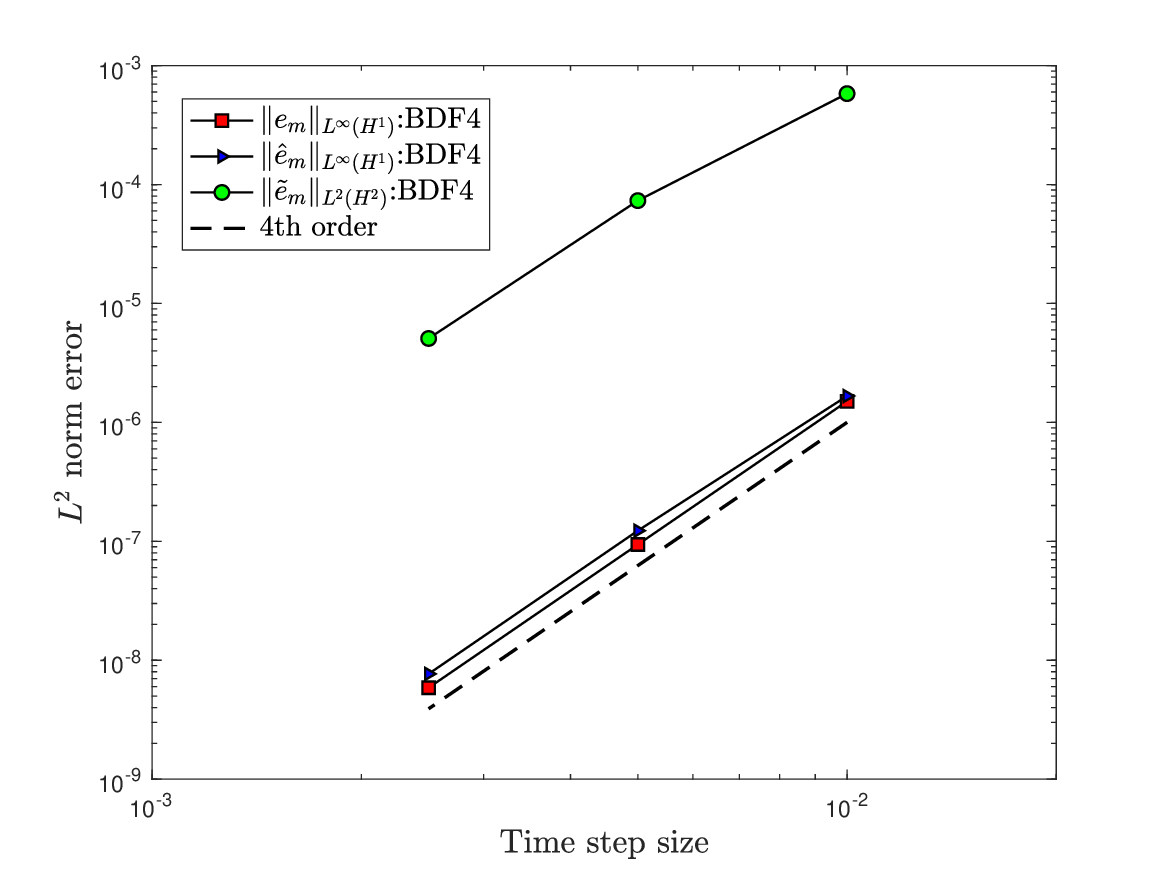}\label{error2d}
					\end{minipage}
			}
	{\caption{Numerical convergence rate of the first- to fourth-order schemes in Example 2.}\label{error2}}
\end{figure}	


\subsection{Phenomenon of blow up} 
As indicated in \cite{chen1998evolution,bartels2008numerical,an2021optimal},  
when the initial value is sufficiently smooth, it is known that \eqref{e_original model} has a unique smooth solution in short time which may blow up at some finite time. So we investigate the possible blowup of the general Landau-Lifshitz equation \eqref{e_original model} with certain smooth initial condition in this subsection. The initial data $\textbf{\textbf{m}}_0$ is defined by
\begin{equation}
\textbf{\textbf{m}}_0(\textbf{x})=	\left \{
	\begin{aligned}
		& (0,0,-1)^T,& t>0,\;|\textbf{x}|\geq \frac{1}{2},
		\\
		&   (\frac{2x_1A}{A^2+|\textbf{x}|^2},
		\frac{2x_2A}{A^2+|\textbf{x}|^2},
		\frac{A^2-|\textbf{x}|^2}{A^2+|\textbf{x}|^2})^T
		& t>0,\;|\textbf{x}|\leq \frac{1}{2},
	\end{aligned}
	\right.
\end{equation}
where $A=(1-2|\textbf{x}|)^4$. We use $128$ Fourier modes in each direction for spatial approximation in $\Omega= [-1/2, 1/2)^2$.
The parameters are set as $\Delta t=10^{-6}, \ S=0.5,\ \beta=1$, $w=1$  and $K_0=0.1$. 


Numerical simulation for the orthogonal projection of the vector field $\textbf{m}^{n+1}$ on the $x_1x_2-$plane and close-up pictures of $\textbf{m}^{n+1}$ near the origin at $t=0,0.01,0.08,0.35,0.501,0.51,0.52, 0.55, 0.6$ are shown in Figures \ref{orthogonal_example2}-\ref{close_example222} by using \eqref{e_High-order1}-\eqref{e_High-order_eta} with $l=2$.  It can be observed that $\textbf{m}^{n+1}$ preserves $(0,0,1)^t$ at the origin and gradually turns down to $(0,0,-1)^t$ near the origin,
which is consistent with the phenomenon of blowup presented in \cite{bartels2008numerical,an2021optimal}.
 We can observe that the energy dissipation law is preserved well in Figure \ref{fig3_example3a} with different Fourier modes $N=32,64,128$. 
In Figure \ref{fig3_example3}, we plot  the evolution of $|\textbf{\textbf{m}}|_{W^{1,\infty}}$ with respect to time, which indicate that the gradient of $\textbf{m}^{n+1}$ near the origin appears to  going to infinity as we refine the mesh. Thus the solution of the Landau-Lifshitz equation \eqref{e_original model} with this smooth initial data may blow up around the origin in a finite time. 
	\begin{figure}[htp]
	\centering 
	\subfigure[$T$=0]{
		\begin{minipage}[c]{0.3\textwidth}
			\includegraphics[width=1\textwidth]{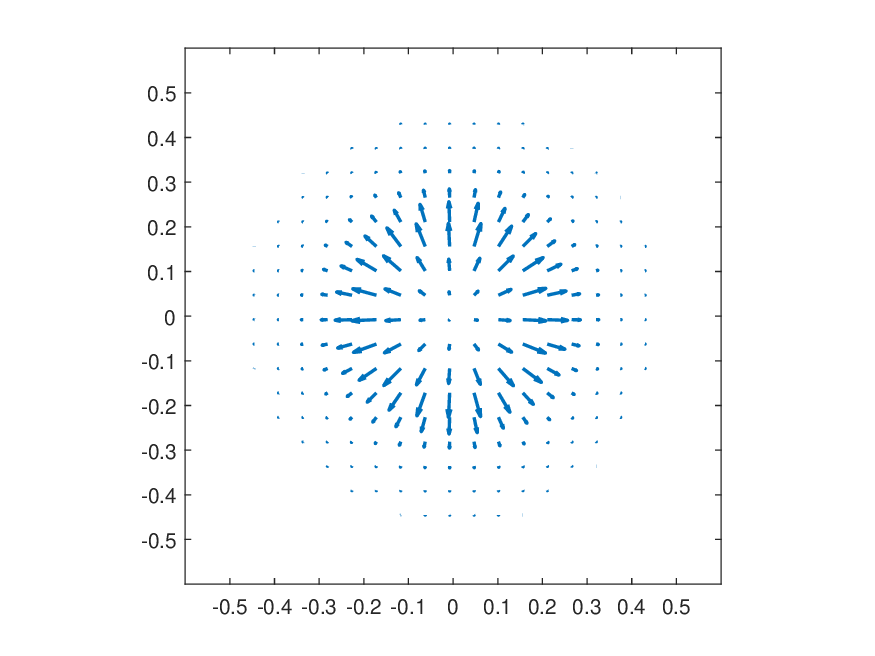}
		\end{minipage}
	}
	\subfigure[$T$=0.01]{
		\begin{minipage}[c]{0.3\textwidth}
			\includegraphics[width=1\textwidth]{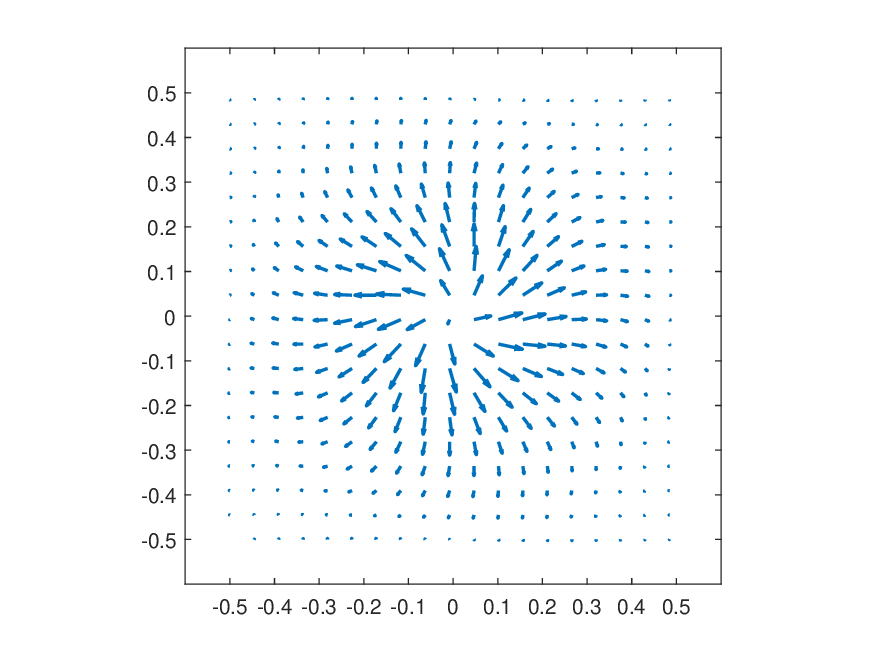}
		\end{minipage}
	}
	\subfigure[$T$=0.08]{
		\begin{minipage}[c]{0.3\textwidth}
			\includegraphics[width=1\textwidth]{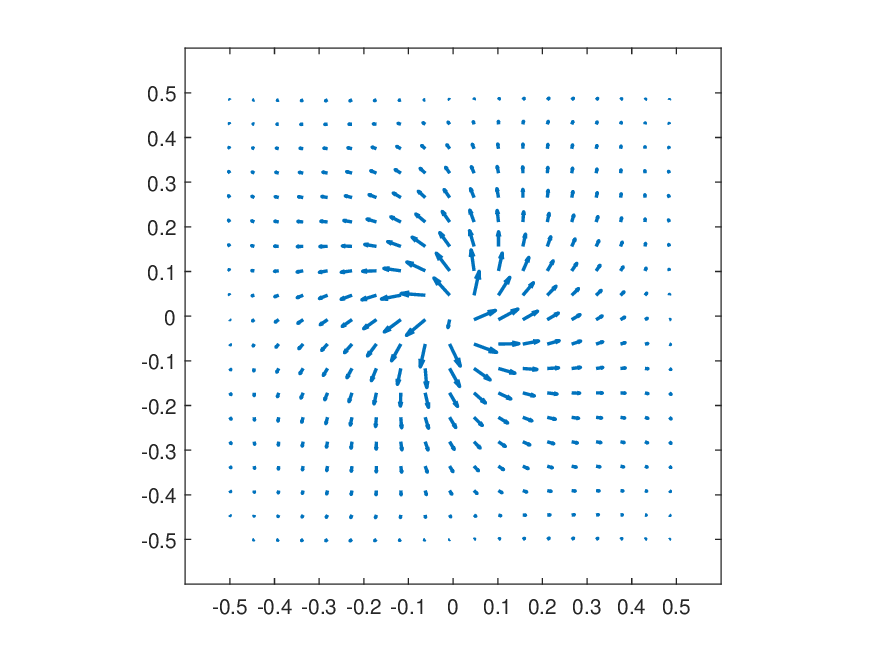}
		\end{minipage}
	}
	\subfigure[$T$=0.35]{
		\begin{minipage}[c]{0.3\textwidth}
			\includegraphics[width=1\textwidth]{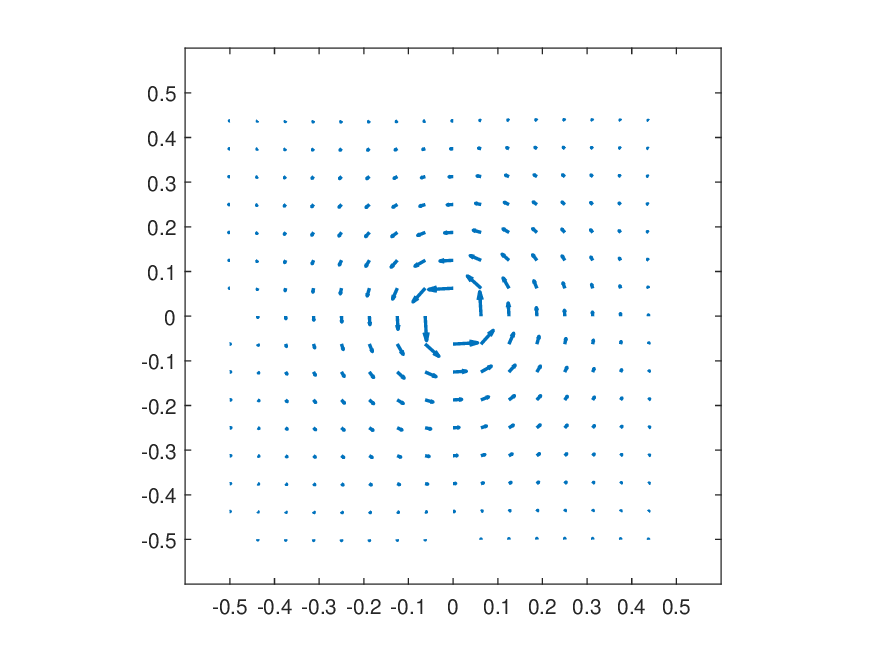}
		\end{minipage}
	}
	\subfigure[$T$=0.501]{
		\begin{minipage}[c]{0.3\textwidth}
			\includegraphics[width=1\textwidth]{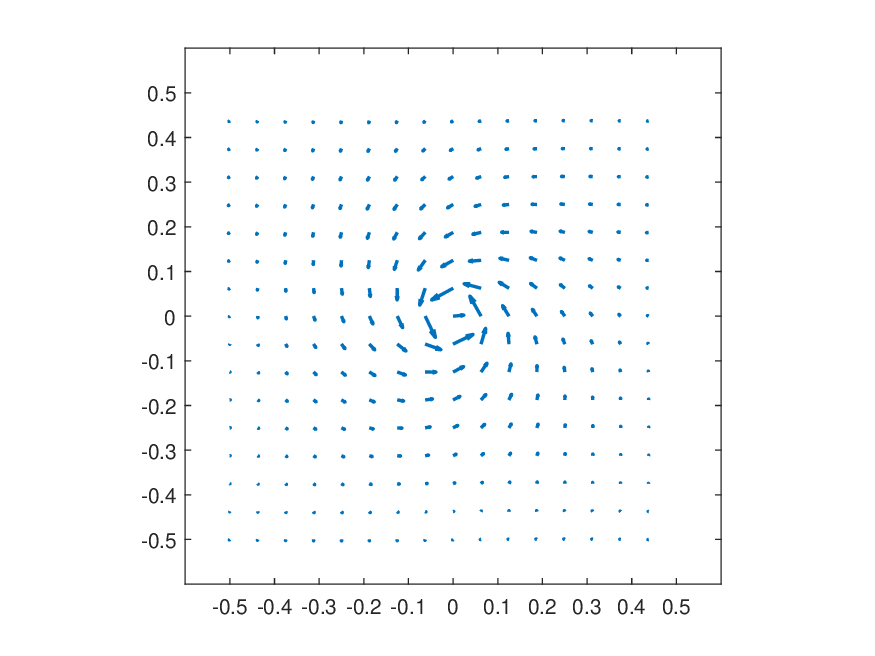}
		\end{minipage}
	}
	\subfigure[$T$=0.51]{
		\begin{minipage}[c]{0.3\textwidth}
			\includegraphics[width=1\textwidth]{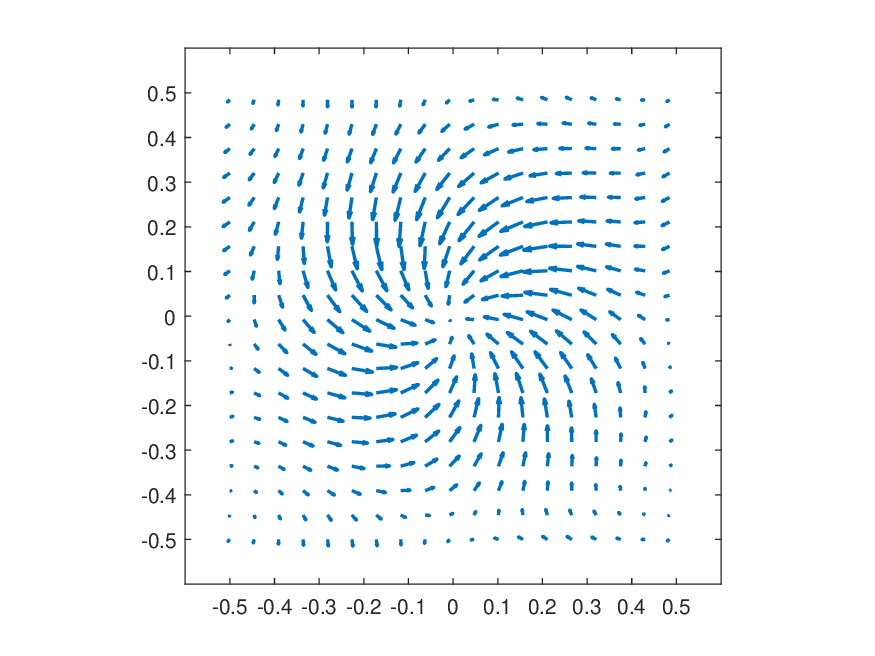}
		\end{minipage}
	}
	\subfigure[$T$=0.52]{
		\begin{minipage}[c]{0.3\textwidth}
			\includegraphics[width=1\textwidth]{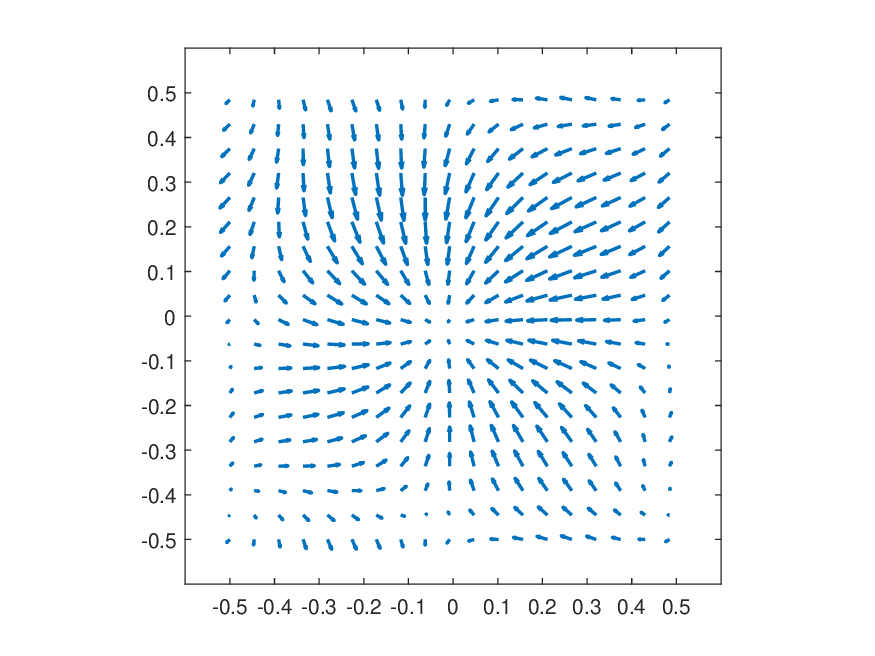}
		\end{minipage}
	}
	\subfigure[$T$=0.55]{
		\begin{minipage}[c]{0.3\textwidth}
			\includegraphics[width=1\textwidth]{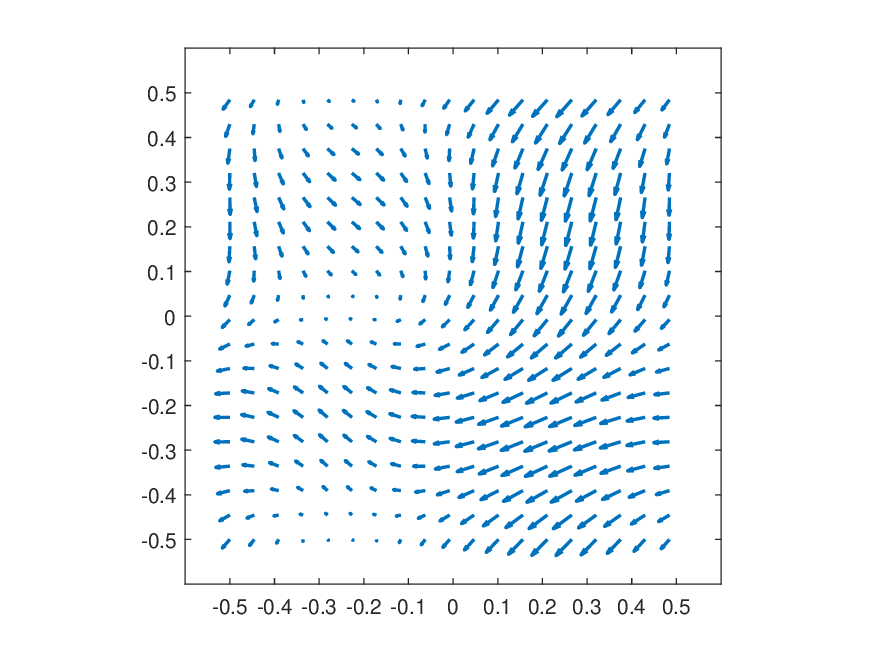}
		\end{minipage}
	}
	\subfigure[$T$=0.6]{
		\begin{minipage}[c]{0.3\textwidth}
			\includegraphics[width=1\textwidth]{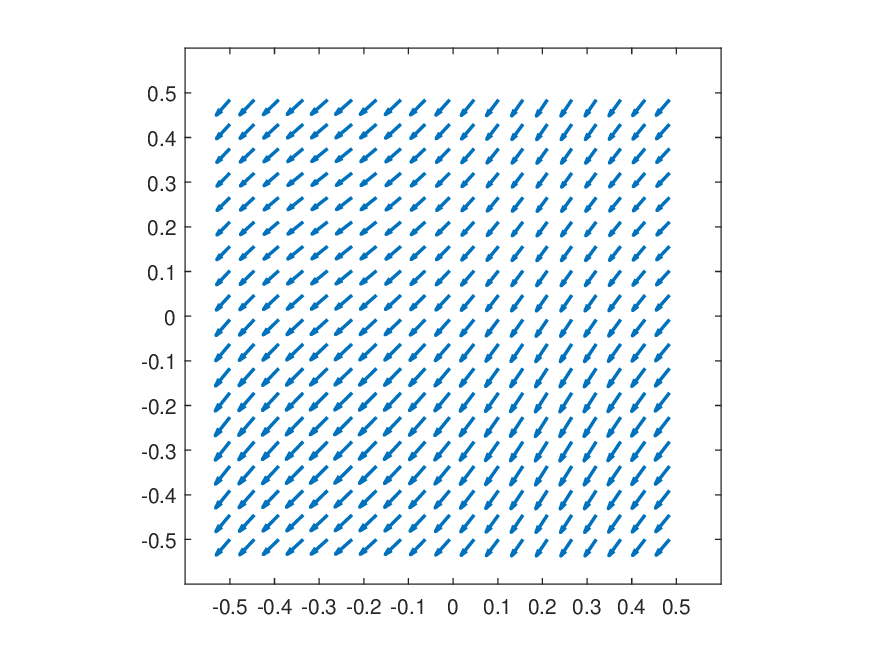}
		\end{minipage}
	}
	{\caption{Numerical magnetization \textbf{\textbf{m}} (projected on $x_1x_2$-plane) with $N=128$.}\label{orthogonal_example2}} 
\end{figure}

\begin{figure}[htp]
	\centering 
	\subfigure[$T$=0.01]{
		\begin{minipage}[c]{0.3\textwidth}
			\includegraphics[width=1.1\textwidth]{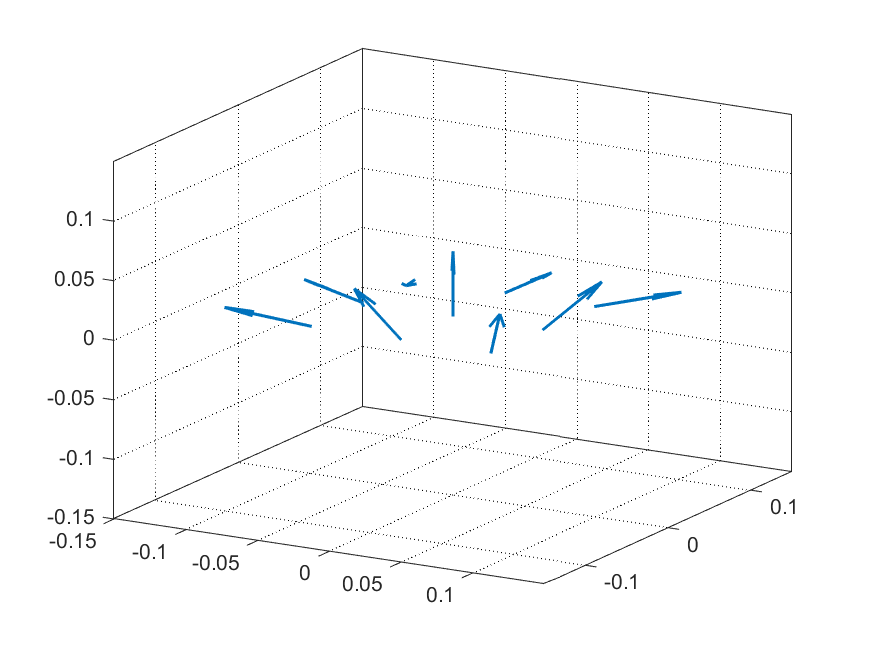}
		\end{minipage}
	}
	\subfigure[$T$=0.08]{
		\begin{minipage}[c]{0.3\textwidth}
			\includegraphics[width=1.1\textwidth]{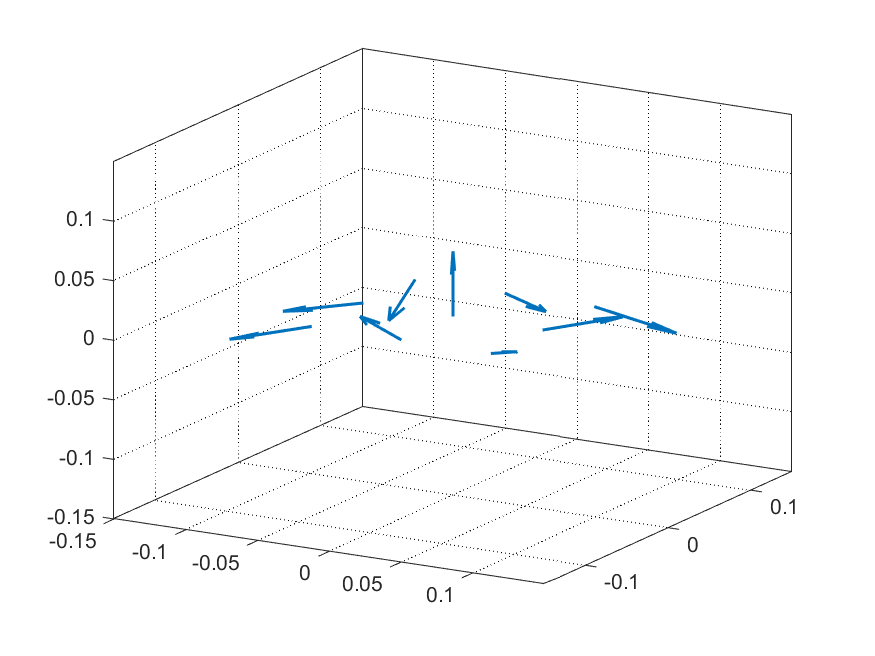}
		\end{minipage}
	}
	\subfigure[$T$=0.35]{
		\begin{minipage}[c]{0.3\textwidth}
			\includegraphics[width=1.1\textwidth]{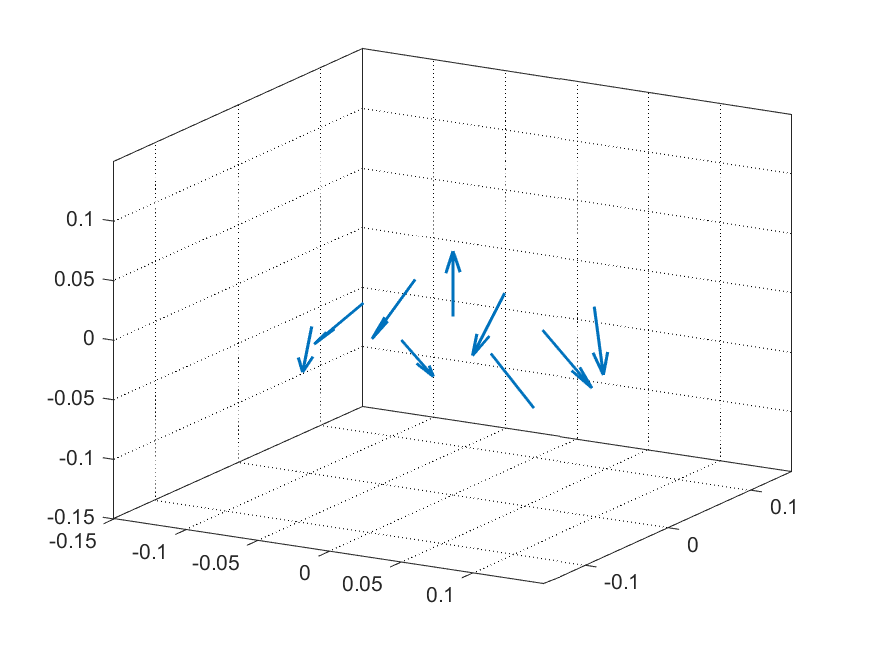}
		\end{minipage}
	}
	\subfigure[$T$=0.501]{
		\begin{minipage}[c]{0.3\textwidth}
			\includegraphics[width=1.1\textwidth]{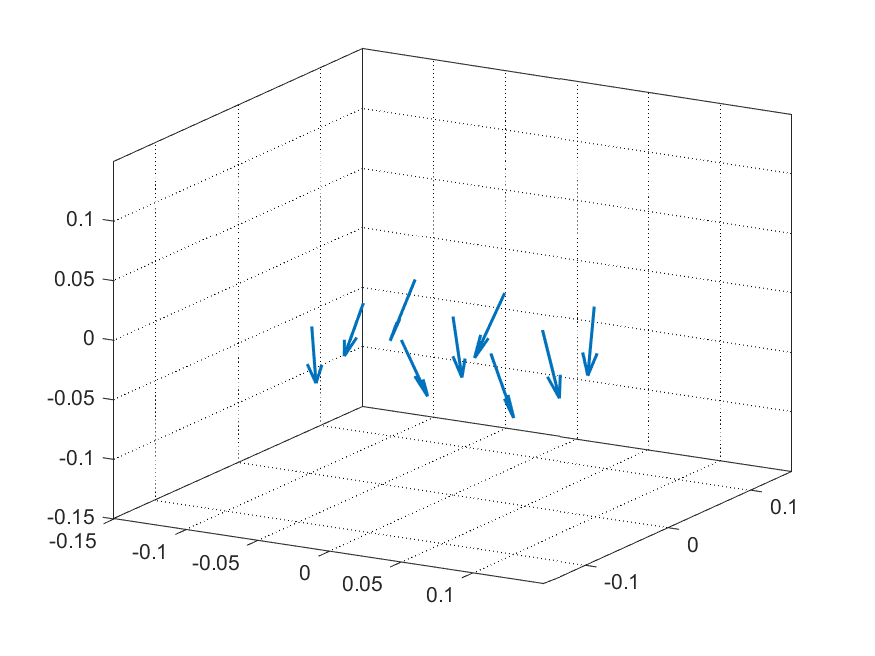}
		\end{minipage}
	}
	\subfigure[$T$=0.51]{
		\begin{minipage}[c]{0.3\textwidth}
			\includegraphics[width=1.1\textwidth]{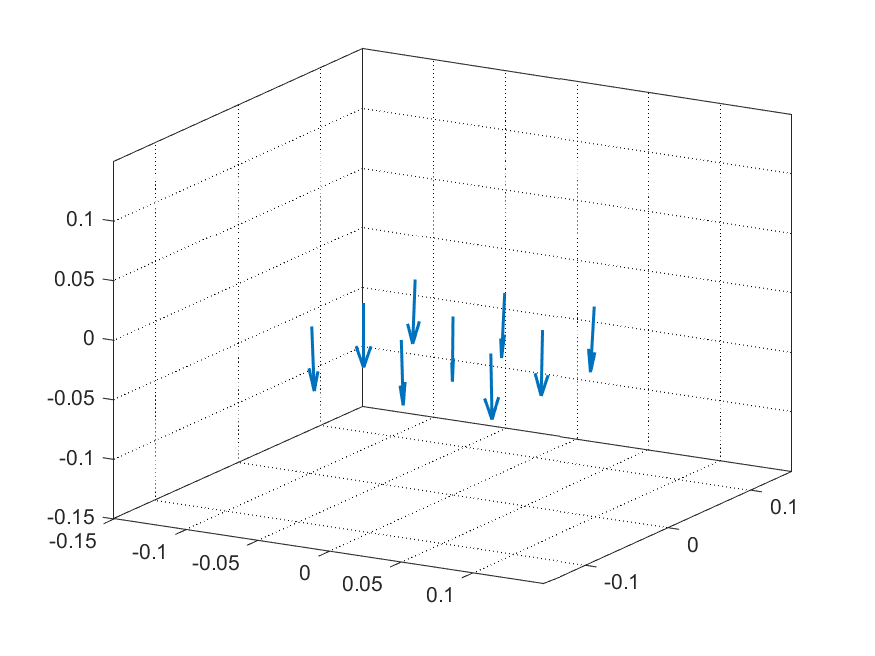}
		\end{minipage}
	}
	\subfigure[$T$=0.6]{
		\begin{minipage}[c]{0.3\textwidth}
			\includegraphics[width=1.1\textwidth]{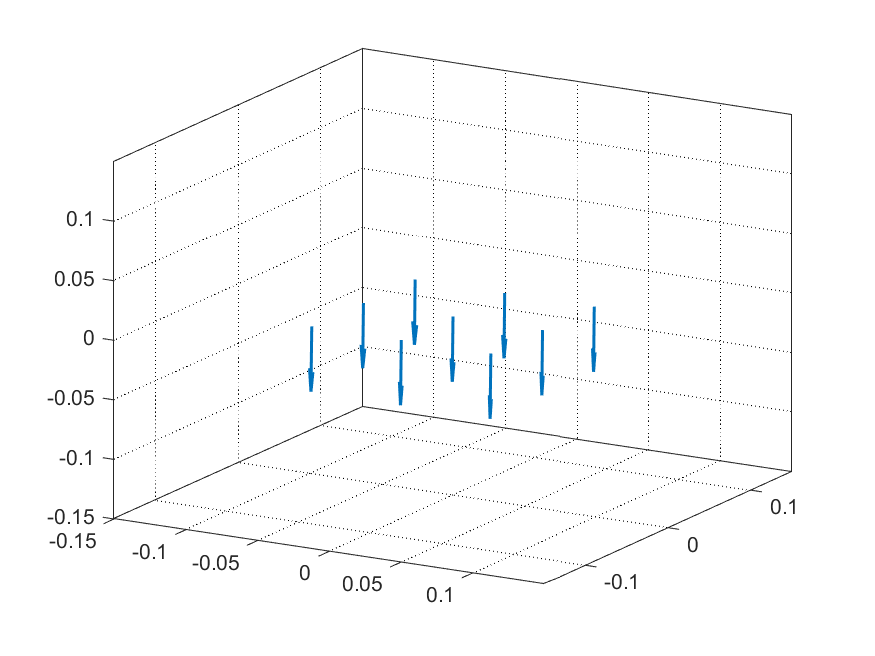}
		\end{minipage}
	}
	{\caption{Numerical magnetization \textbf{\textbf{m}} around the origin with $N=128$.}\label{close_example222}}
\end{figure}

	\begin{figure}[htp]
	\centering 
	\subfigure[]{
		\begin{minipage}[c]{0.45\textwidth}
			\includegraphics[width=1.1\textwidth]{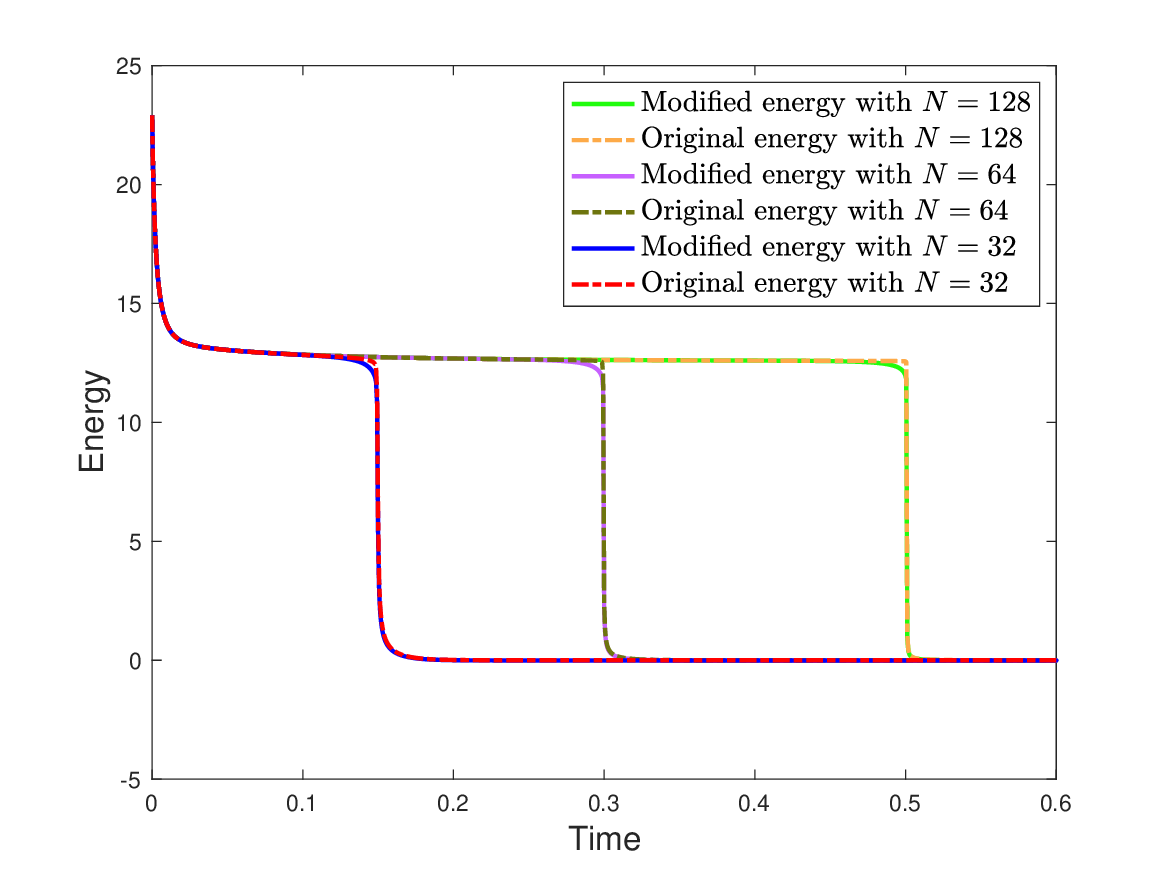}\label{fig3_example3a}
		\end{minipage}
	}
	\subfigure[]{
		\begin{minipage}[c]{0.45\textwidth}
			\includegraphics[width=1.1\textwidth]{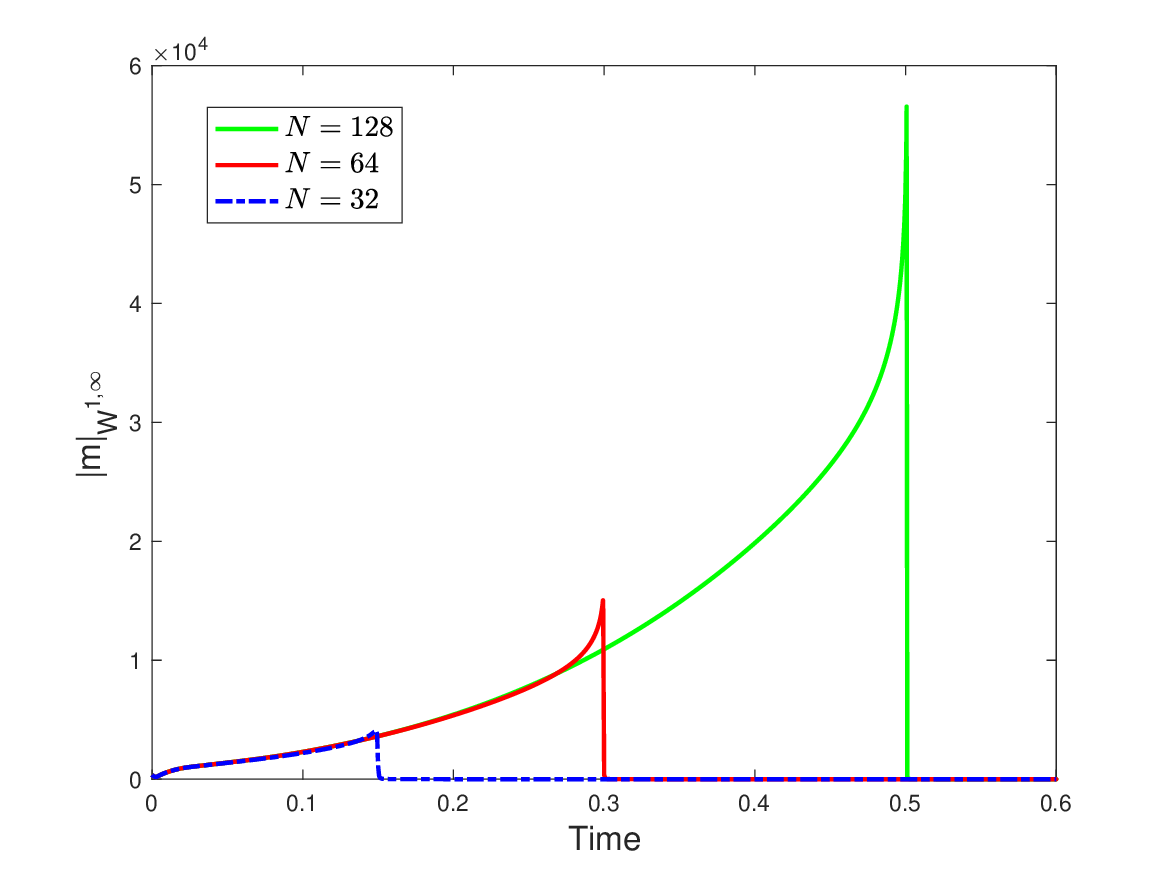}\label{fig3_example3}
		\end{minipage}
	}	{\caption{Evolutions of energy (left) and $|\textbf{\textbf{m}}|_{W^{1,\infty}}$ (right) computed 
		 with different Fourier modes.
	 } }
\end{figure}

\section{Concluding remarks}
We constructed a new class of  high-order implicit-explicit (IMEX) schemes based on the GSAV approach for the Landau-Lifshitz equation. These  schemes are linear, length preserving, and  {\color{black} at each time step only require solving  (a) decoupled  elliptic equations with constant coefficients  when the nonlinear term is treated fully explicitly, or (b) a coupled elliptic system with variable coefficients  when the nonlinear term is treated semi-implicitly.}  Furthermore, their modified energies  are unconditionally decreasing, and their numerical solutions are unconditionally bounded in $l^{\infty}(0,T;H^1(\Omega)) $. We also established rigorous error estimates up to fifth-order  for these schemes
 in $l^{\infty}(0,T;H^1(\Omega)) \bigcap l^{2}(0,T;H^2(\Omega))$. 
 
 To the best of our knowledge, these are the first higher than second-order error estimates  for schemes which enforce normalization of the magnetization. Although we only considered  semi-discrete (in time) case in this paper,  these schemes can be easily implemented  with finite difference or spectral methods, and it is expected that 
the analysis can be extended, albeit tedious, to  suitable fully discrete  approximations.
\bibliographystyle{siamplain}
\bibliography{final}

\end{document}